\documentclass[11pt]{article}

\usepackage[utf8]{inputenc}
\usepackage{xcolor}

\definecolor{yellow1}{rgb}{1,0.8,0.2}

\usepackage[numbers,sort&compress]{natbib}
\bibpunct[, ]{[}{]}{,}{n}{,}{,}

\usepackage{
	amsmath,
	amsthm,
	amsfonts,
	amssymb,
	amsbsy,
	amsopn,
	amstext
}
\usepackage{mathtools}
\usepackage{mathrsfs}

\usepackage{algorithm}
\usepackage{algorithmic}

\usepackage{graphicx}
\usepackage{epstopdf}
\usepackage{float}
\usepackage[caption=false]{subfig}

\usepackage{tikz}
\usetikzlibrary{
	arrows,
	decorations.pathmorphing,
	backgrounds,
	positioning,
	fit,
	petri,
	automata
}
\usepackage{booktabs}
\usepackage{threeparttable}
\usepackage{multirow}
\usepackage{diagbox}
\usepackage{makecell}
\renewcommand{\arraystretch}{1.5}
\usepackage{enumitem}
\usepackage{appendix}
\usepackage{indentfirst}
\usepackage{emptypage}
\usepackage{comment}
\usepackage{geometry}
\usepackage{hyperref}
\hypersetup{hidelinks}
\theoremstyle{plain}
\newtheorem{theorem}{Theorem}[section]
\newtheorem{lem}[theorem]{Lemma}

\theoremstyle{definition}

\theoremstyle{remark}

\newtheorem{ass}{Assumption}

\newif\ifrevision
\revisiontrue

\ifrevision

\else

\fi

\newcommand{\normm}[1]{%
	\left\vert\kern-0.25ex
	\left\vert\kern-0.25ex
	\left\vert #1
	\right\vert\kern-0.25ex
	\right\vert\kern-0.25ex
	\right\vert
}

\begin{document}

\title{An Efficient Stochastic First-Order Algorithm for Nonconvex--Strongly Concave Minimax Optimization beyond Lipschitz Smoothness}

\author{Yan~Gao and Yongchao~Liu\thanks{School of Mathematical Sciences, Dalian University of Technology, Dalian 116024, China, e-mail: gydllg123@mail.dlut.edu.cn (Yan~Gao), lyc@dlut.edu.cn (Yongchao~Liu) }}
\date{}
\maketitle
\noindent{\bf Abstract.} 
In recent years, nonconvex minimax problems have attracted significant attention because of their broad applications in machine learning, including generative adversarial networks, robust optimization and adversarial training. Most existing algorithms for nonconvex stochastic minimax problems are developed under the standard Lipschitz smoothness assumption. In this paper, we study stochastic minimax problems under a generalized smoothness condition and propose an algorithm, NSGDA-M, which simultaneously updates the inner variable by stochastic gradient ascent and updates the outer variable by normalized stochastic gradient descent with momentum. 
When the objective function is nonconvex--strongly concave, we show that NSGDA-M finds an \(\epsilon\)-stationary point of the primal function within \(\mathcal O(\epsilon^{-4}\log(1/(\epsilon\delta)))\)
stochastic gradient evaluations with probability at least \(1-\delta\).
Moreover, we
establish an expected stationarity guarantee of
\(\mathcal O(\delta^{-3/4}T^{-1/4})+G_\Phi\delta\), which gives a convergence rate \(\mathcal O(T^{-1/7})\). Here \(G_\Phi\) bounds the primal gradient norms along the iterates.
Numerical experiments on a distributionally robust optimization problem demonstrate the effectiveness of the proposed algorithm. 

	
	\section{Introduction}
	
	Consider the following nonconvex--strongly concave stochastic minimax problem
	\begin{equation}\label{RS_minimax_problem}
		\min_{x \in \mathbb{R}^n} \max_{y \in \mathcal{Y}} \mathcal{L}(x, y) := \mathbb{E}_{\xi \sim \mathcal{P}} \left[ l(x, y; \xi) \right],
	\end{equation}
	where $\mathcal{Y}\subseteq\mathbb{R}^{m}$ is a closed convex set, $\xi$ is drawn from a possibly unknown probability distribution $\mathcal{P}$, and for each $\xi$, $l(\cdot,\cdot;\xi)$ is nonconvex in $x$ and $\mu$-strongly concave in $y$. 
	Problem \eqref{RS_minimax_problem} has attracted significant attention because of its broad applications in modern machine learning, including generative adversarial networks \cite{Creswell2018Generative, Goodfellow2014Advance}, distributionally robust optimization \cite{Levy2020Large,namkoong2016stochastic} and adversarial training \cite{sinha2018certifiable,tu2019theoretical}.
	
	Over the past decade, numerous algorithms for nonconvex--strongly concave stochastic minimax problems have been developed under the classical Lipschitz smoothness assumption.
	A natural method is gradient descent with a max-oracle (GDmax)~\cite{jin2020local},
	which is a double-loop algorithm that alternates between a gradient descent step in $x$ and an (approximate) maximization step in $y$. Lin et al.~\cite{lin2020gradient} extend the GDmax to stochastic setting (SGDmax) and show that SGDmax finds an $\epsilon$-stationary point within $\tilde{\mathcal{O}}\left(\epsilon^{-4}\right)$ stochastic gradient evaluations.\footnote{The notation
		\(\widetilde{\mathcal O}(\cdot)\) suppresses logarithmic factors.} 
	Several subsequent studies improve the convergence rate by using inexact proximal-point frameworks~\cite{rafique2022weakly,zhang2022sapd+}, variance-reduction techniques~\cite{luo2020stochastic,xu2021enhanced}, and adaptive-stepsize setting~\cite{yang2022nest}.
	On the other hand,
	Lin et al.~\cite{lin2020gradient} propose a single loop algorithm, stochastic gradient descent ascent algorithm (SGDA) that alternates between a stochastic gradient descent step in \(x\) and a stochastic gradient ascent step in \(y\), in which the stepsize $\eta_y$ is larger than $\eta_x$ to force $y$ to move faster than $x$, 
	achieving an $\mathcal{O}(\epsilon^{-4})$ stochastic gradient complexity. 
	SGDA has a simple structure, and many variants of SGDA have been proposed, including alternating SGDA~\cite{boct2023alternating}, momentum-accelerated SGDA~\cite{huang2022accelerated,huang2023adagda}, smoothed SGDA~\cite{yang2022faster}, adaptive-stepsize SGDA~\cite{xu2024stochastic,li2023tiada}, and distributed SGDA~\cite{chen2024efficient,sharma2022federated,wu2024solving}.
	
	However, the classical Lipschitz smoothness assumption may be violated in modern machine learning applications, including distributionally robust optimization~\cite{jin2021nonconvex} and neural networks \cite{graves2012long, crawshaw2022robustness}, or may hold only with a prohibitively large Lipschitz constant $L$, resulting in overly conservative complexity bounds. 
	Recently, Zhang et al.~\cite{zhang2020why} propose the \((L_0, L_1)\)-smoothness condition for minimization problems, which relaxes the global Lipschitz bound by permitting the Hessian norm to scale linearly with the local gradient norm.  Further extensions of this generalized smoothness framework can be found in~\cite{chen2023generalized,li2023convex}.
	To accommodate generalized smoothness, clipping and normalization techniques have been incorporated into stochastic first-order methods~\citep{faw2023beyond,wang2023convergence,
		crawshaw2022robustness,hubler2024parameter,
		li2023convergence,liu2023near,
		reisizadeh2025variance,yang2025independent}.
	For generalized smoothness minimax problems, there are two challenges:
	(i) characterizing
	the descent of a Lyapunov function under generalized smoothness, and (ii) ensuring that the iterates remain in a 
	region where the local smoothness condition applies.
	Recently, Xian et al.~\cite{xian2024delving} study deterministic and
	stochastic variants of GDA and GDmax under the
	\((\ell_x,\ell_y)\)-smoothness condition. 
	They develop a trajectory-wise analysis that controls the primal gradients and the maximizer tracking error
	along the iterates; in the stochastic setting, they use a stopping-time
	localization argument to establish high-probability guarantees.
	They show that
	generalized GDA and GDmax achieve an
	\(\mathcal O(\epsilon^{-2})\) iteration complexity, while their
	stochastic variants achieve an
	\(\mathcal O(\delta^{-4}\epsilon^{-4})\) stochastic gradient complexity with probability at least \(1-\delta\).
	\footnote{The dependence on \(\delta\) is not explicitly displayed in the complexity notation of \cite{xian2024delving}; it is made explicit here for comparison.}
	For deterministic generalized smooth minimax problems, 
	Zhu et al.~\cite{zhu2025finding} propose an
	accelerated negative curvature gradient descent ascent algorithm
	(ANCGDA) with only gradient evaluations involved and establish an
	\(\mathcal{O}(\epsilon^{-1.75}\log n)\) complexity for finding
	second-order stationary points under a generalized Hessian smoothness
	condition.

	We contribute to minimax optimization under generalized smoothness by
	proposing NSGDA-M, which simultaneously updates the inner variable by a
	stochastic gradient ascent step and the outer variable by a normalized
	stochastic gradient descent step with momentum. 
	NSGDA-M exploits the normalized momentum update to directly control
	primal movement, reducing the reliance on trajectory-wise gradient
	magnitude control required by existing generalized smooth minimax
	analyses~\cite{xian2024delving,zhu2025finding}.
	Building on the stopping-time localization technique in
	\cite{xian2024delving}, we develop a stopped-process analysis framework
	and establish stationarity guarantees both with high probability and in
	expectation.
	When the objective function is nonconvex--strongly concave and satisfies the
	\((L_0,L_1)\)-generalized smoothness condition, NSGDA-M finds an
	\(\epsilon\)-stationary point of the primal function within
	\(\mathcal O(\epsilon^{-4}\log(1/(\epsilon\delta)))\) stochastic gradient
	evaluations with probability at least \(1-\delta\). 
	Compared with the high-probability complexity bound established in
	Xian et al.~\cite{xian2024delving}, NSGDA-M improves the dependence on
	the failure probability \(\delta\) from polynomial to logarithmic with
	constant batch size. Moreover, we establish an expected stationarity guarantee of \(\mathcal O(\delta^{-3/4}T^{-1/4})+G_\Phi\delta\), which gives a convergence rate \(\mathcal O(T^{-1/7})\).
	Here \(G_\Phi\) bounds the primal gradient norms along the iterates.
	The high-probability and expected stationarity analyses treat the localization failure event differently: the former controls the probability of leaving the localized region, whereas the latter also accounts for the effect of the failure event through an additional bounded-gradient condition. The framework extends to 
	nonconvex--strongly concave minimax problems with compact convex inner constraints under a slightly stronger assumption.

	The structure of this paper is organized as follows.
	Section \ref{RS_Section_2} introduces the NSGDA-M algorithm and presents necessary assumptions on problem \eqref{RS_minimax_problem}. Section~\ref{RS_Section_3} presents the convergence analysis of NSGDA-M, where Subsection~\ref{RS_Subsection_3_1} establishes the high-probability stationarity guarantee, and Subsection~\ref{RS_Subsection_3_2} establishes the expected stationarity guarantees. Section \ref{RS_Section_4} verifies the effectiveness of the proposed algorithm through numerical experiments on a distributionally robust optimization problem.

	\noindent\textbf{Notation.} 
	$\mathbb{R}^n$ denotes the $n$-dimensional Euclidean space equipped with the norm $\|x\| = \sqrt{\langle x, x \rangle}$. $\mathbb{N}$ denotes the set of nonnegative integers. 
	For \(a,b\in\mathbb R\), let \(a\wedge b:=\min\{a,b\}\).
	For a differentiable function $f(\cdot)$, $\nabla f$ denotes its full gradient, $\nabla_x f(\cdot)$ and $\nabla_y f(\cdot)$ denote the partial gradients with respect to $x$ and $y$, respectively.  $\operatorname{proj}_{\mathcal{Y}}(y)$ denotes the Euclidean projection of $y$ onto the set $\mathcal{Y} \subseteq \mathbb{R}^m$.  
	We denote $a = \mathcal{O}(b)$ if there exists a constant $C > 0$ such that $|a| \le C |b|$ and $[M] := \{1, 2, \dots, M\}$. 
	Throughout this paper, for any vector \(v\), the normalized direction
	\(v/\|v\|\) is understood as \(\mathbf 0\) when \(v=\mathbf 0\).

	\section{Preliminaries and Algorithm}\label{RS_Section_2}
	In this section, we present the NSGDA-M algorithm and some necessary assumptions. 
	
	\begin{algorithm}[htbp]
		\caption{Normalized Stochastic Gradient Descent Ascent with Momentum (NSGDA-M)}
		\label{algorithm:NSGDA-M}
		\begin{algorithmic}[1]
			\STATE \textbf{Input:} $T$, stepsizes $\eta_x, \eta_y$, momentum parameter $\beta\in (0,1)$.
			\STATE \textbf{Initialization:}
			\(x^{0}\in\mathbb{R}^{n}\), \(y^{0}\in\mathcal Y\), and \(m^0=\mathbf{0}\).
			\FOR{$t=0$ \TO $T-1$}
			\STATE Draw a stochastic sample $\xi^t \sim \mathcal{P}$.
			\STATE Compute stochastic gradients $G_{x}\left(x^t, y^t, \xi^t\right),\, G_{y}\left(x^t, y^t, \xi^t\right).$
			\STATE Update the primal variable $x$:
			{\small
				\begin{align}
					&m^{t+1}=\beta m^{t}+(1-\beta)G_x(x^t,y^t,\xi^t), \label{RS_equ:alg_m}\\
					&x^{t+1}=x^{t}-\eta_{x}\frac{m^{t+1}}{\|m^{t+1}\|}. \label{RS_equ:alg_x}
			\end{align}}
			\STATE Update the dual variable $y$:
			{\small
				\begin{align}
					&y^{t+1}=\operatorname{proj}_{\mathcal{Y}}\left(y^{t}+\eta_{y}G_y(x^t,y^t,\xi^t)\right). \label{RS_equ:alg_y}
			\end{align}}
			\ENDFOR
			\STATE \textbf{Output:} \(\widehat{x}\) chosen uniformly at random from
			\(\{x^t\}_{t=0}^{T-1}\).
		\end{algorithmic}
	\end{algorithm}
	
	In Algorithm~\ref{algorithm:NSGDA-M}, Step 5 computes the stochastic gradient estimates of the objective function $\mathcal{L}(\cdot)$. Step 6 updates the primal variable $x$ by first performing a momentum step in~\eqref{RS_equ:alg_m}, and then applying a normalized stochastic gradient descent update in~\eqref{RS_equ:alg_x}. The
	normalized update controls the magnitude of the primal update
	independently of the gradient magnitude, while the momentum term smooths
	the stochastic gradient direction by averaging historical gradient
	information.
	Normalized stochastic gradient methods with momentum are widely used in nonconvex stochastic optimization, and it has been shown that the incorporation of the additional momentum mechanism can remove the need for large batch sizes for nonconvex objectives \cite{Cutkosky2020MomentumNSGD, Sun2023MomentumSIGNSGD}. 
	Step 7 updates the dual variable \(y\) by a projected stochastic gradient
	ascent step. This update applies to both dual-domain settings
	considered in this paper: the unconstrained setting
	\(\mathcal{Y}=\mathbb{R}^m\) and the constrained setting in which
	\(\mathcal{Y}\) is a compact convex set.
	Unlike the generalized SGDA and SGDmax methods in~\cite{xian2024delving} which require batch sizes of order $\Theta(\epsilon^{-2})$ to ensure convergence, NSGDA-M converges with a constant batch size.

	We next introduce the notation and assumptions used in our convergence
	analysis:
	\begin{gather*}
		\Phi(\cdot)\coloneqq
		\max_{y\in\mathcal Y}\mathcal L(\cdot,y),
		\qquad
		y^*(\cdot)\in
		\operatorname*{argmax}_{y\in\mathcal Y}
		\mathcal L(\cdot,y),
		\qquad
		\Phi^*\coloneqq
		\inf_{x\in\mathbb R^n}\Phi(x).
	\end{gather*}
	The filtration $\{\mathcal{F}_t\}_{t\geq 0}$ is defined by  
	$\mathcal{F}_t = \sigma(\xi^0, \xi^1, \dots, \xi^{t-1})$ for $t \geq 1$ 
	and $\mathcal{F}_0 = \{\emptyset, \Omega\}$.
	
	\begin{ass}\label{RS_ass_PhiLowerBound}
		The primal function $\Phi(\cdot)$ is lower bounded, i.e., $\inf_x \Phi(x) = \Phi^* > -\infty$.
	\end{ass}
	\begin{ass}\label{RS_ass_stronglyConcave_Y}
		The objective function and feasible set satisfy
		
		(a) For any $x\in \mathbb{R}^n$, $\mathcal{L}(x,\cdot)$ is $\mu$-strongly concave with $\mu>0$.
		
		(b) The feasible set $\mathcal{Y}$ is either $\mathbb{R}^m$ or a nonempty compact convex subset of $\mathbb{R}^m$.
	\end{ass}
	\begin{ass}[Generalized Smoothness]\label{RS_ass_L0L1Smooth}
		For the function
		\(\mathcal{L}(\cdot):\mathbb{R}^n\times\mathbb{R}^m
		\rightarrow \mathbb{R}\),
		there exist constants
		\(L_{x,0},L_{x,1},L_{y,0},L_{y,1}> 0\)
		such that, for any two points
		\(\mathbf{u}=(x,y)\) and
		\(\mathbf{u}'=(x',y')\) satisfying
		\(
		\|\mathbf{u}-\mathbf{u}'\|
		\leq
		\frac{1}{L_{x,1}}
		+
		\frac{1}{L_{y,1}},
		\)
		the following inequalities hold
		\begin{align*}
			\bigl\|
			\nabla_x \mathcal{L}(\mathbf{u})
			-
			\nabla_x \mathcal{L}(\mathbf{u}')
			\bigr\|
			&\leq
			\bigl(
			L_{x,0}
			+
			L_{x,1}
			\|\nabla_x \mathcal{L}(\mathbf{u})\|
			\bigr)
			\|\mathbf{u}-\mathbf{u}'\|,
			\\
			\bigl\|
			\nabla_y \mathcal{L}(\mathbf{u})
			-
			\nabla_y \mathcal{L}(\mathbf{u}')
			\bigr\|
			&\leq
			\bigl(
			L_{y,0}
			+
			L_{y,1}
			\|\nabla_y \mathcal{L}(\mathbf{u})\|
			\bigr)
			\|\mathbf{u}-\mathbf{u}'\|.
		\end{align*}
	\end{ass}
	
	\begin{ass}\label{RS_ass_gradyStarBound}
		There exists a constant $B \geq 0$ such that for all $x \in \mathbb{R}^n$,
		\[
		\|\nabla_{y}\mathcal{L}(x, y^{*}(x))\| \leq B.
		\]
	\end{ass}
	
	\begin{ass}
		\label{RR_ass:unbiased}
		For every $t \geq 0$, the stochastic gradients satisfy
		\begin{align*}
			\mathbb{E}[G_x(x^t,y^t,\xi^t) \mid \mathcal{F}_t] &= \nabla_x \mathcal{L}(x^t,y^t),\\
			\mathbb{E}[G_y(x^t,y^t,\xi^t) \mid \mathcal{F}_t] &= \nabla_y \mathcal{L}(x^t,y^t).
		\end{align*}
	\end{ass}
	
	Assumption~\ref{RS_ass_PhiLowerBound} is the standard lower-boundedness condition in nonconvex optimization.
	Assumption~\ref{RS_ass_stronglyConcave_Y}~(a) ensures the uniqueness
	of the maximizer \(y^*(x)\) for each fixed \(x\).
	Assumption~\ref{RS_ass_stronglyConcave_Y}~(b) allows the inner
	maximization problem to be either unconstrained or constrained over
	a nonempty compact convex set.
	Assumption~\ref{RS_ass_L0L1Smooth} is an
	\((L_0,L_1)\)-generalized smoothness condition on \(\mathcal L(\cdot)\), as in
	\cite{zhang2020why,xian2024delving}. It extends the standard
	Lipschitz smoothness condition by allowing the local Lipschitz
	constants to depend on the corresponding partial-gradient norms.	
	Assumption~\ref{RS_ass_gradyStarBound} imposes a uniform bound on
	the \(y\)-partial gradient evaluated at the inner maximizer. It holds with $B=0$ when $\mathcal{Y}=\mathbb{R}^m$ and serves as a regularity condition when \(\mathcal Y\) is a compact convex set. Together with Assumption~\ref{RS_ass_L0L1Smooth}, it yields the uniform local
	\(y\)-smoothness constant \(L_y:=L_{y,0}+L_{y,1}B\) used in the
	\(y\)-tracking analysis. 
	Assumption~\ref{RR_ass:unbiased} states the conditional unbiasedness
	of the stochastic gradients, a standard assumption in stochastic
	optimization.
	Throughout
	the convergence analysis, we set
	\begin{equation}\label{RS_equ:conditionNum_def}
		L_y:=L_{y,0}+L_{y,1}B,
		\qquad
		\kappa:=\frac{L_y}{\mu}.
	\end{equation}

	\section{Convergence Analysis}\label{RS_Section_3}
	In this section, we establish convergence guarantees for NSGDA-M. We
	first present several auxiliary lemmas used in the subsequent analysis.
	We then establish the high-probability and expected stationarity
	guarantees in Subsections~\ref{RS_Subsection_3_1} and
	\ref{RS_Subsection_3_2}, respectively.
	
	\begin{lem}\label{RS_lem_LipSolution}
		Suppose Assumptions \ref{RS_ass_PhiLowerBound}--\ref{RS_ass_gradyStarBound} hold. For any $x$ and $x'$ that satisfy $\|x' - x\| \leq \frac{1}{L_{x,1}}$, we have
		\begin{equation*}
			\|y^*(x) - y^*(x')\| \leq \kappa \|x - x'\|.
		\end{equation*}
	\end{lem}
	
	\begin{proof}
		Since $y^{*}(\cdot)=\operatorname{argmax}_{y\in\mathcal{Y}}\mathcal{L}(\cdot, y)$ for any $x\in\mathbb{R}^n$, we have
		\begin{equation*}
			\langle y - y^*(x), \nabla_y \mathcal{L}(x, y^*(x)) \rangle \leq 0, \quad \langle y - y^*(x'), \nabla_y \mathcal{L}(x', y^*(x')) \rangle \leq 0.
		\end{equation*}
		Substituting \(y=y^*(x')\) into the first inequality
		and \(y=y^*(x)\) into the second one, and then summing
		the two resulting inequalities, we obtain
		\begin{equation*}
			\langle y^*(x) - y^*(x'), \nabla_y \mathcal{L}(x', y^*(x')) - \nabla_y \mathcal{L}(x, y^*(x)) \rangle \leq 0.
		\end{equation*}
		Since $\mathcal{L}(\cdot)$ is $\mu$-strongly concave with respect to $y$, we have
		\begin{equation*}
			\mu \|y^*(x) - y^*(x')\|^2 \leq \langle y^*(x) - y^*(x'), \nabla_y \mathcal{L}(x', y^*(x')) - \nabla_y \mathcal{L}(x', y^*(x)) \rangle.
		\end{equation*}
		Combining the above two inequalities, we have
		\begin{equation*}
			\begin{aligned}
				\mu \|y^*(x) - y^*(x')\|^2 \leq{}& \langle y^*(x) - y^*(x'), \nabla_y \mathcal{L}(x, y^*(x)) - \nabla_y \mathcal{L}(x', y^*(x)) \rangle\\
				\leq{}& \|y^*(x) - y^*(x')\| \|\nabla_y \mathcal{L}(x, y^*(x)) - \nabla_y \mathcal{L}(x', y^*(x))\|.
			\end{aligned}
		\end{equation*}
		
		Given $\|x' - x\| \leq \frac{1}{L_{x,1}}$, we have by Assumption \ref{RS_ass_L0L1Smooth} that
		\begin{equation*}
			\begin{aligned}
				\|\nabla_y \mathcal{L}(x, y^*(x)) - \nabla_y \mathcal{L}(x', y^*(x))\| \leq{}& (L_{y,0} + L_{y,1}\|\nabla_y \mathcal{L}(x, y^*(x))\|) \|x - x'\|\\
				\leq{}& (L_{y,0} + L_{y,1}B)\|x - x'\|,
			\end{aligned}
		\end{equation*}
		where the last inequality follows from Assumption~\ref{RS_ass_gradyStarBound}. Note that $B=0$ when $\mathcal{Y} = \mathbb{R}^m$.	
		
		Then, we have
		\begin{equation*}
			\begin{aligned}
				\|y^*(x) - y^*(x')\| \leq{}& \frac{(L_{y,0} + L_{y,1}B)}{\mu} \|x - x'\|.
			\end{aligned}
		\end{equation*}
		The proof is complete.
	\end{proof}
	
	\begin{lem}\label{RS_lem_primalSmoothness}
		Suppose Assumptions \ref{RS_ass_PhiLowerBound}--\ref{RS_ass_gradyStarBound} hold. For any $x$ and $x'$ that satisfy $\|x'-x\| \leq \frac{1}{(1+\kappa)L_{x,1}}$, we have
		\begin{align}
			&\|\nabla \Phi(x) - \nabla \Phi(x')\| \leq (1 + \kappa)(L_{x,0} + L_{x,1}\|\nabla \Phi(x)\|) \|x - x'\|, \tag{a}\\
			&\Phi(x') \leq \Phi(x) + \langle \nabla \Phi(x), x' - x \rangle + \frac{1 + \kappa}{2}(L_{x,0} + L_{x,1}\|\nabla \Phi(x)\|) \|x - x'\|^2, \tag{b}\\
			&\Phi(x') \geq \Phi(x) + \langle \nabla \Phi(x), x' - x \rangle - \frac{1 + \kappa}{2}(L_{x,0} + L_{x,1}\|\nabla \Phi(x)\|) \|x - x'\|^2. \tag{c}
		\end{align}
	\end{lem}
	\begin{proof}
		By Assumption \ref{RS_ass_L0L1Smooth} and Lemma \ref{RS_lem_LipSolution}, we have
		\begin{equation*}
			\begin{aligned}
				\|\nabla \Phi(x') - \nabla \Phi(x)\| = &\|\nabla_x \mathcal{L}(x', y^*(x')) - \nabla_x \mathcal{L}(x, y^*(x))\|\\
				\leq &(L_{x,0} + L_{x,1}\|\nabla \Phi(x)\|) (\|x' - x\| + \|y^*(x') - y^*(x)\|)\\
				\leq &(1 + \kappa)(L_{x,0} + L_{x,1}\|\nabla \Phi(x)\|) \|x' - x\|,
			\end{aligned}
		\end{equation*}
		which verifies (a).
		
		Let $z(t) = x + t(x' - x), t\in [0,1]$.
		By the Newton–Leibniz formula, we have
		\begin{equation*}
			\Phi(x') = \Phi(x) + \langle \nabla \Phi(x), x' - x \rangle + \int_0^1 \langle \nabla \Phi(z(t)) - \nabla \Phi(x), x' - x \rangle dt.
		\end{equation*}
		Then, we have by the Cauchy–Schwarz inequality that
		\begin{equation*}
			\begin{aligned}
				\left|\Phi(x') - \Phi(x) - \langle \nabla \Phi(x), x' - x \rangle\right| \leq{} &(1+\kappa) (L_{x,0} + L_{x,1}\|\nabla \Phi(x)\|) \|x' - x\|^2 \int_0^1 t dt\\
				\leq{}&\frac{1+\kappa}{2} (L_{x,0} + L_{x,1}\|\nabla \Phi(x)\|) \|x - x'\|^2,
			\end{aligned}
		\end{equation*}
		which implies (b)-(c). The proof is complete.
	\end{proof}

	\subsection{Convergence analysis of NSGDA-M in high probability}\label{RS_Subsection_3_1}
	This subsection establishes a high-probability stationarity guarantee for Algorithm~\ref{algorithm:NSGDA-M}. The analysis relies on a stopping-time localization argument. We first state the stochastic
	gradient noise assumption and the auxiliary lemmas used below.
	
	\begin{ass}
		\label{RS_ass_gradienError_proba}
		There exist constants \(\sigma_x,\sigma_y>0\) such that,
		for every \(t\geq0\),
		the stochastic gradient noise satisfies the 
		following conditions:
		\begin{enumerate}[label=(\alph*)]
			\item $\|G_x(x^t, y^t, \xi^t) - \nabla_x \mathcal{L}(x^t, y^t)\| \leq \sigma_x$, almost surely;
			
			\item $\mathbb{P}\!\left(\|G_y(x^t, y^t, \xi^t) - \nabla_y \mathcal{L}(x^t, y^t)\| \geq \lambda \mid \mathcal{F}_t\right) 
			\leq 2 \exp\!\left(-\frac{\lambda^{2}}{2\sigma_y^{2}}\right)$, for all $\lambda \geq 0$.
		\end{enumerate}
	\end{ass}
	
	Assumption~\ref{RS_ass_gradienError_proba}(a) provides the
	almost-sure bound on the primal stochastic gradient noise required
	by the martingale concentration argument in
	Lemma~\ref{RS_lem_martingale}, whereas Assumption~
	\ref{RS_ass_gradienError_proba}(b) imposes a conditional
	norm-sub-Gaussian tail bound on the dual stochastic gradient noise
	used to establish the conditional MGF bounds required by
	Lemma~\ref{RS_lem_RecursiveControlMGF}.

	\begin{lem}\label{RS_lem_martingale}\cite[Lemma 2.4]{liu2023near}
		Suppose $X_1, \ldots, X_T$ is a martingale difference sequence adapted to a filtration $\mathcal{F}_1, \mathcal{F}_2, \ldots ,\mathcal{F}_T$ in a Hilbert space such that $\|X_t\| \leq R_t, \forall t \in [T]$ almost surely for some constants $R_t \geq 0$. Then for any given $\delta \in (0,1)$, with probability at least $1 - \delta$, for all $t \in [T]$ we have
		\begin{equation*}
			\left\|\sum_{s=1}^{t} X_s\right\| \leq 4 \sqrt{\log(2/\delta) \sum_{s=1}^{T} R_s^2}.
		\end{equation*}
	\end{lem}
	\begin{lem}
		[Recursive control on MGF,
		{\cite[Proposition~29]{Cutler2023Stochastic}}]
		\label{RS_lem_RecursiveControlMGF}
		Consider scalar stochastic processes $(V_t)$, $(U_t)$, and $(X_t)$ on a probability space with filtration $(\mathcal{H}_t)$ that are linked by the inequality
		\begin{equation}\label{RS_equ_RecursiveControlMGF_recursion}
			V_{t+1} \leq \alpha_t V_t + U_t \sqrt{V_t} + X_t + \kappa_t
		\end{equation}
		for some deterministic constants $\alpha_t \in (-\infty, 1]$ and $\kappa_t \in \mathbb{R}$. Suppose the following properties hold:
		
		(a) $V_t$ is nonnegative and $\mathcal{H}_t$-measurable.
		
		(b) $U_t$ is mean-zero sub-Gaussian conditioned on $\mathcal{H}_t$ with deterministic parameter $\sigma_t$:
		\[
		\mathbb{E}[\exp(\lambda U_t) \mid \mathcal{H}_t] \leq \exp(\lambda^2 \sigma_t^2 / 2), \quad \forall \lambda \in \mathbb{R}.
		\]
		
		(c) $X_t$ is nonnegative and sub-exponential conditioned on $\mathcal{H}_t$ with deterministic parameter $\nu_t$:
		\[
		\mathbb{E}[\exp(\lambda X_t) \mid \mathcal{H}_t] \leq \exp(\lambda \nu_t), \quad \forall \lambda \in [0, 1/\nu_t].
		\]
		
		Then the estimate
		\[
		\mathbb{E}[\exp(\lambda V_{t+1})] \leq \exp(\lambda (\nu_t + \kappa_t)) \mathbb{E}[\exp(\lambda (1 + \alpha_t) V_t / 2)]
		\]
		holds for any $\lambda$ satisfying $0 \leq \lambda \leq \min \left( \frac{1 - \alpha_t}{2 \sigma_t^2}, \frac{1}{2 \nu_t} \right)$.
	\end{lem}
	
	\begin{theorem}
		\label{RS_thm_HP_NSGDAM}\label{RS_thm_NSGDA-M_highProba}
		Let \(L_{x,0},L_{x,1},L_{y,0},L_{y,1}\) denote the generalized
		smoothness constants, \(\sigma_x,\sigma_y\) denote the noise
		parameters, and \(T\) denote the iteration horizon. Choose
		\[
		\begin{aligned}
			1-\beta
			=
			\frac{1}{\sqrt{T\ell_T}},
			\quad
			\eta_y
			=
			\frac{1-\beta}{\mu}
			\left(
			\frac{4\kappa^2\mu^2}
			{121\sigma_y^2(1+\kappa)^2L_{x,0}^2}
			\right)^{1/3},\quad
			\eta_x
			=
			\frac{1-\beta}{(1+\kappa)L_{x,0}}
			\left(
			\frac{\ell_T}{T}
			\right)^{1/4},
		\end{aligned}
		\]
		where $L_y$ and $\kappa$ are defined in \eqref{RS_equ:conditionNum_def}, \(\delta\in(0,1)\) such that \(\ell_T:=\log\frac{2T}{\delta}\), and \(T\geq2\). 
		Suppose that \(T\) is sufficiently large such that
		\[
		\eta_y
		\leq
		\min\left\{
		\frac{1}{5\mu},
		\,
		\frac{\mu}{4L_y^2},
		\,
		\frac{\mu}
		{
			18
			\left(
			\frac{3751}{10}+\frac{22}{\log4}
			\right)
			\sigma_y^2L_{x,1}^2\ell_T
		}
		\right\},\quad
		\eta_x
		\leq
		\frac{1-\beta}
		{8(1+\kappa)L_{x,1}}.
		\]
		Suppose also that the initial tracking error satisfies \[\|y^0-y^*(x^0)\| \leq \frac{1}{16L_{x,1}}, \] and Assumptions
		\ref{RS_ass_PhiLowerBound}--\ref{RS_ass_gradienError_proba}
		hold. Then there exists a constant \(C>0\), independent of \(T\) and
		\(\delta\), such that
		\[
		\frac1T
		\sum_{t=0}^{T-1}
		\|\nabla\Phi(x^t)\|
		\leq
		C
		\left(
		\frac{\ell_T}{T}
		\right)^{1/4}
		\]
		holds with probability at least \(1-\delta\).
	\end{theorem}
	
	\begin{proof}
		Since \(\Phi(\cdot)\) is \((L_0,L_1)\)-smooth with \(L_0:=(1+\kappa)L_{x,0}\) and \(L_1:=(1+\kappa)L_{x,1}\), we have by the update rule of \(x\) in \eqref{RS_equ:alg_x} and \(\eta_x\leq \frac{1}{L_1}\) that
		\[
		\begin{aligned}
			\Phi(x^{t+1})
			\leq{}&
			\Phi(x^t)
			+
			\left\langle
			\nabla\Phi(x^t),
			x^{t+1}-x^t
			\right\rangle
			+
			\frac{L_0+L_1\|\nabla\Phi(x^t)\|}{2}
			\|x^{t+1}-x^t\|^2
			\\
			\leq{}&
			\Phi(x^t)
			-
			\eta_x
			\left\langle
			\nabla\Phi(x^t),
			\frac{m^{t+1}}{\|m^{t+1}\|}
			\right\rangle
			+
			\frac{\eta_x^2}{2}
			\left(
			L_0+L_1\|\nabla\Phi(x^t)\|
			\right)
			\\
			={}&
			\Phi(x^t)
			-
			\eta_x
			\left\langle
			m^{t+1}-\varepsilon_t,
			\frac{m^{t+1}}{\|m^{t+1}\|}
			\right\rangle
			+
			\frac{\eta_x^2}{2}
			\left(
			L_0+L_1\|\nabla\Phi(x^t)\|
			\right)
			\\
			\leq{}&
			\Phi(x^t)
			-
			\eta_x\|m^{t+1}\|
			+
			\eta_x\|\varepsilon_t\|
			+
			\frac{\eta_x^2}{2}
			\left(
			L_0+L_1\|\nabla\Phi(x^t)\|
			\right)
			\\
			\leq{}&
			\Phi(x^t)
			-
			\eta_x\|\nabla\Phi(x^t)\|
			+
			2\eta_x\|\varepsilon_t\|
			+
			\frac{\eta_x^2}{2}
			\left(
			L_0+L_1\|\nabla\Phi(x^t)\|
			\right),
		\end{aligned}
		\]
		where \(\varepsilon_t:=m^{t+1}-\nabla\Phi(x^t)\), the second inequality follows from
		Cauchy--Schwarz inequality, and the last inequality follows from the fact
		\[
		\|m^{t+1}\|
		=
		\|\nabla\Phi(x^t)+\varepsilon_t\|
		\geq
		\|\nabla\Phi(x^t)\|
		-
		\|\varepsilon_t\|.
		\]
		
		Summing the above inequality over \(t=0,\ldots,T-1\) yields
		\begin{equation}
			\label{RS_equ_HP_descent}
			\left(
			1-\frac{\eta_xL_1}{2}
			\right)
			\frac{1}{T}
			\sum_{t=0}^{T-1}
			\|\nabla\Phi(x^t)\|
			\leq
			\frac{2}{T}
			\sum_{t=0}^{T-1}\|\varepsilon_t\|
			+
			\frac{\Phi(x^0)-\Phi^*}{T\eta_x}
			+
			\frac{\eta_xL_0}{2}.
		\end{equation}
		For the first term on the right-hand side of \eqref{RS_equ_HP_descent}, 
		we have by the definition of \(\varepsilon_t\) and the update rule of $m^t$ in \eqref{RS_equ:alg_m},
		\[
		\begin{aligned}
			\varepsilon_{t+1}
			={}&
			m^{t+2}-\nabla\Phi(x^{t+1})                         \\
			={}&
			\beta m^{t+1}
			+
			(1-\beta)G_x(x^{t+1},y^{t+1},\xi^{t+1})
			-
			\nabla\Phi(x^{t+1})                                  \\
			={}&
			\beta\varepsilon_t
			+
			\beta
			\left(
			\nabla\Phi(x^t)-\nabla\Phi(x^{t+1})
			\right)
			+
			(1-\beta)\widehat{\varepsilon}_{t+1},\quad \forall t\geq0,
		\end{aligned}
		\]
		where 
		\begin{equation}
			\label{eq:gradient_estimator_error}
			\widehat{\varepsilon}_t
			:=
			G_x(x^t,y^t,\xi^t)-\nabla\Phi(x^t).
		\end{equation}
		and the initial term is given by
		\[
		\varepsilon_0
		=
		m^1-\nabla\Phi(x^0)
		=
		\beta
		\left(
		m^0-\nabla\Phi(x^0)
		\right)
		+
		(1-\beta)\widehat{\varepsilon}_0.
		\]	
		Since
		\(\|x^{i+1}-x^i\|\leq\eta_x\leq1/L_1\),
		Lemma~\ref{RS_lem_primalSmoothness}(a) applies. Iterating the
		preceding recursion with the resulting bound yields
		\begin{equation}\label{eq:momentum_error_bound}
			\begin{aligned}
				\|\varepsilon_t\|
				\leq{}&
				\beta^{t+1}
				\|m^0-\nabla\Phi(x^0)\|                              
				+
				\eta_x\beta
				\sum_{i=0}^{t-1}
				\beta^{t-i-1}
				\left(
				L_0+L_1\|\nabla\Phi(x^i)\|
				\right)                                                \\
				&+
				(1-\beta)
				\left\|
				\sum_{i=0}^{t}
				\beta^{t-i}
				\widehat{\varepsilon}_i
				\right\|.
			\end{aligned}
		\end{equation}
		Summing the above inequality over \(t=0,\ldots,T-1\) and applying the
		triangle inequality, we obtain
		\begin{equation}\label{RS_equ_HP_momentum_error_bound}
			\begin{aligned}
				\sum_{t=0}^{T-1}
				\|\varepsilon_t\|
				\leq{}&
				\underbrace{
					(1-\beta)
					\sum_{t=0}^{T-1}
					\left\|
					\sum_{i=0}^{t}
					\beta^{t-i}
					\Delta_x^i
					\right\|}_{:=I_1}+
				\frac{\eta_xL_1\beta}{1-\beta}
				\sum_{t=0}^{T-1}
				\|\nabla\Phi(x^t)\|
				\\
				&\underbrace{+
					(1-\beta)
					\sum_{t=0}^{T-1}
					\sum_{i=0}^{t}
					\beta^{t-i}
					\left\|
					\nabla_x\mathcal L(x^i,y^i)
					-
					\nabla\Phi(x^i)
					\right\|}_{:=I_2}
				\\
				&+
				\frac{T\eta_xL_0\beta}{1-\beta}
				+
				\frac{\beta}{1-\beta}
				\|m^0-\nabla\Phi(x^0)\|,
			\end{aligned}
		\end{equation}
		where \(\Delta_x^i:=G_x(x^i,y^i,\xi^i)-\nabla_x\mathcal L(x^i,y^i).\)

		To bound \(I_1\), for each \(i=0,\ldots,T-1\), define
		\(Z_i\in(\mathbb{R}^n)^T\) by
		\[
		(Z_i)_r
		:=
		\begin{cases}
			\beta^{r-i}\Delta_x^i,
			& i\leq r\leq T-1,\\
			\mathbf{0},
			& 0\leq r<i,
		\end{cases}
		\qquad
		r=0,\ldots,T-1,
		\]
		where \(\beta\in (0,1)\), \(\mathbf{0}\in\mathbb{R}^n\), and \(\Delta_x^i\in\mathbb{R}^n\).
		Note that \(\Delta_x^i\) is
		\(\mathcal F_{i+1}\)-measurable, and hence \(Z_i\) is
		\(\mathcal F_{i+1}\)-measurable. Moreover, by
		Assumption~\ref{RR_ass:unbiased},
		\[
		\mathbb E
		\left[
		Z_i
		\,\middle|\,
		\mathcal F_i
		\right]
		=
		0.
		\]
		Thus, \((Z_i)_{i=0}^{T-1}\) is a martingale difference
		sequence adapted to \((\mathcal F_j)_{j=1}^{T}\).
		Furthermore, by Assumption~\ref{RS_ass_gradienError_proba}(a) and
		\(0<\beta<1\),
		\[
		\begin{aligned}
			\|Z_i\|^2
			=
			\sum_{r=i}^{T-1}
			\beta^{2(r-i)}
			\|\Delta_x^i\|^2
			\leq
			\sigma_x^2
			\sum_{r=i}^{T-1}
			\beta^{2(r-i)}
			\leq
			\frac{\sigma_x^2}{1-\beta^2}
			\qquad
			\text{a.s.}
		\end{aligned}
		\]
		Applying Lemma~\ref{RS_lem_martingale} with confidence parameter \(\delta/2\) and
		the deterministic bounds
		\[
		R_i
		=
		\frac{\sigma_x}{\sqrt{1-\beta^2}},
		\qquad
		i=0,\ldots,T-1,
		\]
		we obtain
		\[
		\mathbb P(\mathcal E_x)
		\geq
		1-\frac{\delta}{2},
		\]
		where
		\[
		\mathcal E_x
		:=
		\left\{
		\left\|
		\sum_{i=0}^{T-1}Z_i
		\right\|
		\leq
		4\sigma_x
		\sqrt{
			\frac{T\log(4/\delta)}
			{1-\beta^2}
		}
		\right\}.
		\]
		
		Note that, for each \(t=0,\ldots,T-1\), the \(t\)-th coordinate of
		\(\sum_{i=0}^{T-1}Z_i\) is given by
		\[
		\left(
		\sum_{i=0}^{T-1}Z_i
		\right)_t
		=
		\sum_{i=0}^{t}
		\beta^{t-i}\Delta_x^i.
		\]
		Thus, when the event \(\mathcal E_x\) holds,
		\begin{equation*}
			\begin{aligned}
				\sum_{t=0}^{T-1}
				\left\|
				\sum_{i=0}^{t}
				\beta^{t-i}\Delta_x^i
				\right\|^2
				=
				\left\|
				\sum_{i=0}^{T-1}Z_i
				\right\|^2
				\leq
				\frac{
					16T\sigma_x^2\log(4/\delta)
				}{
					1-\beta^2
				}.
			\end{aligned}
		\end{equation*}
		
		Therefore, by the Cauchy--Schwarz inequality,
		\[
		\begin{aligned}
			I_1
			&\leq
			(1-\beta)\sqrt{T}
			\left(
			\sum_{t=0}^{T-1}
			\left\|
			\sum_{i=0}^{t}
			\beta^{t-i}\Delta_x^i
			\right\|^2
			\right)^{1/2}
			\\
			&\leq
			4T\sigma_x
			\sqrt{
				\frac{1-\beta}{1+\beta}
				\log\frac{4}{\delta}
			}.
		\end{aligned}
		\]

		To bound \(I_2\), we define the stopping time 
		\[
		\tau
		:=
		\inf\left\{
		t\ge0:
		\|y^t-y^*(x^t)\|
		>
		\frac{1}{4L_{x,1}}
		\right\},
		\qquad
		\inf\emptyset:=+\infty.
		\]
		First, we bound the dual tracking error $\|y^t - y^*(x^t)\|^2$ for $t < \tau$. By the update rule for \(y\), the optimality of \(y^*(x^t)\),
		and the nonexpansiveness of projection, we have
		\begin{equation}
			\label{RS_equ_HP_y_track_1}
			\begin{aligned}
				&
				\left\|y^{t+1}-y^*(x^t)\right\|^2
				\\
				={}&
				\left\|
				\operatorname{proj}_{\mathcal Y}
				\left(
				y^t+\eta_yG_y(x^t,y^t,\xi^t)
				\right)
				-
				\operatorname{proj}_{\mathcal Y}
				\left(
				y^*(x^t)
				+
				\eta_y\nabla_y\mathcal L(x^t,y^*(x^t))
				\right)
				\right\|^2
				\\
				\le{}&
				\left\|
				y^t-y^*(x^t)
				+
				\eta_y
				\left(
				G_y(x^t,y^t,\xi^t)
				-
				\nabla_y\mathcal L(x^t,y^*(x^t))
				\right)
				\right\|^2
				\\
				\le{}&
				\left\|y^t-y^*(x^t)\right\|^2
				+
				2\eta_y
				\left\langle
				\nabla_y\mathcal L(x^t,y^t)
				-
				\nabla_y\mathcal L(x^t,y^*(x^t)),
				y^t-y^*(x^t)
				\right\rangle
				\\
				&+
				2\eta_y^2
				\left\|
				\nabla_y\mathcal L(x^t,y^t)
				-
				\nabla_y\mathcal L(x^t,y^*(x^t))
				\right\|^2
				\\
				&+
				2\eta_y
				\left\langle
				G_y(x^t,y^t,\xi^t)
				-
				\nabla_y\mathcal L(x^t,y^t),
				y^t-y^*(x^t)
				\right\rangle
				\\
				&+
				2\eta_y^2
				\left\|
				G_y(x^t,y^t,\xi^t)
				-
				\nabla_y\mathcal L(x^t,y^t)
				\right\|^2 .
			\end{aligned}
		\end{equation}

		For the second term on the right-hand side of
		\eqref{RS_equ_HP_y_track_1}, by the \(\mu\)-strong concavity of
		\(\mathcal L(x^t,\cdot)\),
		\[
		2\eta_y
		\left\langle
		\nabla_y\mathcal L(x^t,y^t)
		-
		\nabla_y\mathcal L(x^t,y^*(x^t)),
		y^t-y^*(x^t)
		\right\rangle
		\le
		-2\mu\eta_y
		\left\|y^t-y^*(x^t)\right\|^2 .
		\]
		For the third term on the right-hand side of \eqref{RS_equ_HP_y_track_1}, 
		Assumptions~\ref{RS_ass_L0L1Smooth} and~\ref{RS_ass_gradyStarBound} imply
		\[
		\begin{aligned}
			&2\eta_y^2
			\left\|
			\nabla_y\mathcal L(x^t,y^t)
			-
			\nabla_y\mathcal L(x^t,y^*(x^t))
			\right\|^2 \\
			\leq{}&
			2\eta_y^2
			\left(
			L_{y,0}
			+
			L_{y,1}
			\|\nabla_y\mathcal L(x^t,y^*(x^t))\|
			\right)^2
			\|y^t-y^*(x^t)\|^2                                      \\
			\leq{}&
			2\eta_y^2L_y^2
			\left\|y^t-y^*(x^t)\right\|^2.
		\end{aligned}
		\]
		
		Substituting the preceding two bounds into
		\eqref{RS_equ_HP_y_track_1} and using the fact \(\eta_y\leq\mu/(4L_y^2)\) yields
		\begin{equation}\label{RS_equ_HP_y_track_2}
			\begin{aligned}
				\left\|y^{t+1}-y^*(x^t)\right\|^2
				\leq &
				\left(
				1-\frac{3\mu\eta_y}{2}
				\right)
				\left\|y^t-y^*(x^t)\right\|^2
				\\
				&+
				2\eta_y
				\left\langle
				G_y(x^t,y^t,\xi^t)
				-
				\nabla_y\mathcal L(x^t,y^t),
				y^t-y^*(x^t)
				\right\rangle
				\\
				&+
				2\eta_y^2
				\left\|
				G_y(x^t,y^t,\xi^t)
				-
				\nabla_y\mathcal L(x^t,y^t)
				\right\|^2.
			\end{aligned}
		\end{equation}
		
		By Young's inequality, we have
		\begin{equation}\label{RS_equ_HP_stopped_recursion}
			\begin{aligned}
				&
				\left\|y^{t+1}-y^*(x^{t+1})\right\|^2
				\\
				\le{}&
				\left(
				1+\frac{\mu\eta_y}{2}
				\right)
				\left\|y^{t+1}-y^*(x^t)\right\|^2
				+
				\left(
				1+\frac{2}{\mu\eta_y}
				\right)
				\left\|y^*(x^{t+1})-y^*(x^t)\right\|^2
				\\
				\le{}&
				\left(
				1-\mu\eta_y
				\right)
				\left\|y^t-y^*(x^t)\right\|^2
				\\
				&+
				2\left(
				1+\frac{\mu\eta_y}{2}
				\right)\eta_y
				\left\langle
				G_y(x^t,y^t,\xi^t)
				-
				\nabla_y\mathcal L(x^t,y^t),
				y^t-y^*(x^t)
				\right\rangle
				\\
				&+
				2\left(
				1+\frac{\mu\eta_y}{2}
				\right)\eta_y^2
				\left\|
				G_y(x^t,y^t,\xi^t)
				-
				\nabla_y\mathcal L(x^t,y^t)
				\right\|^2
				+
				\left(
				1+\frac{2}{\mu\eta_y}
				\right)
				\kappa^2\eta_x^2 ,
			\end{aligned}
		\end{equation}
		where the second inequality follows from
		\eqref{RS_equ_HP_y_track_2},
		Lemma~\ref{RS_lem_LipSolution}, and
		\(\|x^{t+1}-x^t\|\leq\eta_x\).

		To extend the localized recursion \eqref{RS_equ_HP_stopped_recursion} to the whole horizon
		\(t=0,\ldots,T-1\), we define the stopped processes
		\[
		\widetilde V_t
		:=
		\|y^t-y^*(x^t)\|^2
		\mathbf 1_{\{t\leq\tau\}}
		\]	
		and 
		\[
		\widetilde\omega_t
		:=
		\begin{cases}
			\displaystyle
			\frac{y^t-y^*(x^t)}
			{\|y^t-y^*(x^t)\|}
			\mathbf 1_{\{t<\tau\}},
			& y^t\neq y^*(x^t),
			\\[1.2ex]
			0,
			& y^t=y^*(x^t).
		\end{cases}
		\]
		By  the definition of the stopped processes and \eqref{RS_equ_HP_stopped_recursion}, we obtain, for \(t=0,\ldots,T-1\),
		\begin{equation}
			\label{RS_equ_HP_stopped_process_recursion}
			\widetilde V_{t+1}
			\leq
			\alpha_t\widetilde V_t
			+
			U_t\sqrt{\widetilde V_t}
			+
			X_t
			+
			\kappa_t,
		\end{equation}
		with
		\[
		U_t
		:=
		2\left(
		1+\frac{\mu\eta_y}{2}
		\right)\eta_y
		\left\langle
		\Delta_y^t,\widetilde\omega_t
		\right\rangle,
		\qquad
		X_t
		:=
		2\left(
		1+\frac{\mu\eta_y}{2}
		\right)\eta_y^2
		\|\Delta_y^t\|^2
		\mathbf 1_{\{t<\tau\}},
		\]
		and
		\[
		\alpha_t:=1-\mu\eta_y,\qquad
		\kappa_t
		:=
		\left(
		1+\frac{2}{\mu\eta_y}
		\right)
		\kappa^2\eta_x^2,
		\]
		where
		\(
		\Delta_y^t
		:=
		G_y(x^t,y^t,\xi^t)
		-
		\nabla_y\mathcal L(x^t,y^t).
		\)
		The parameter choice of \(\eta_y\) gives
		\[
		1-\mu\eta_y\in(0,1).
		\]
		Note that \eqref{RS_equ_HP_stopped_process_recursion} is in the form of \eqref{RS_equ_RecursiveControlMGF_recursion} in Lemma~\ref{RS_lem_RecursiveControlMGF}.

		Next, we verify conditions (a)--(c) of Lemma~\ref{RS_lem_RecursiveControlMGF}.
		Obviously, \(\widetilde V_t\) is nonnegative and \(\mathcal F_t\)-measurable, so condition~(a) holds.
		Since \(\widetilde\omega_t\) is
		\(\mathcal F_t\)-measurable and Assumption~\ref{RR_ass:unbiased} holds,
		we have
		\[
		\mathbb E[U_t\mid\mathcal F_t]=0.
		\] 
		By Assumptions~\ref{RR_ass:unbiased} and~\ref{RS_ass_gradienError_proba}(b), conditioned on
		\(\mathcal F_t\), \(\Delta_y^t\) is a mean-zero
		norm-sub-Gaussian random vector with parameter \(\sigma_y\).
		Then, we have by \cite[Lemma~3]{jin2019short} and \(\|\widetilde\omega_t\|\leq1\) that,  conditioned on \(\mathcal F_t\), \(U_t\) is sub-Gaussian with parameter 
		\(\sigma_t
		=
		2\eta_y
		\left(
		1+\frac{\mu\eta_y}{2}
		\right)c\sigma_y\), and 
		\(
		X_t
		\)
		is sub-exponential with parameter
		\(\nu_t
		=
		2\eta_y^2
		\left(
		1+\frac{\mu\eta_y}{2}
		\right)c\sigma_y^2\). 
		Here, \(c\geq1\) is a universal tail-to-MGF conversion constant; we
		take \(c=5\), which is admissible for both the projection and
		squared-norm MGF bounds.
		Thus, conditions~(b) and~(c) hold.
		
		Then, Lemma~\ref{RS_lem_RecursiveControlMGF} implies, for any
		\(0\leq \lambda\leq \min\left\{\frac{1-\alpha_t}{2\sigma_t^2},\frac{1}{2\nu_t}\right\}\),
		\begin{equation}\label{RS_equ_HP_MGF_recursive_bound}
			\mathbb E
			\left[
			\exp(\lambda\widetilde V_{t+1})
			\right]
			\leq
			\exp\left(\lambda(\nu_t+\kappa_t)\right)
			\mathbb E
			\left[
			\exp
			\left(
			\lambda\alpha\widetilde V_t
			\right)
			\right],
		\end{equation}
		where \(
		\alpha
		:=
		\frac{1+\alpha_t}{2}
		=
		1-\frac{\mu\eta_y}{2}.
		\)
		
		By the definitions of \(\alpha_t\), \(\sigma_t\), and \(\nu_t\),
		\[
		\lambda_0
		:=
		\frac{\mu}
		{200
			\left(
			1+\frac{\mu\eta_y}{2}
			\right)^2
			\eta_y\sigma_y^2}
		\leq
		\min
		\left\{
		\frac{1-\alpha_t}{2\sigma_t^2},
		\frac{1}{2\nu_t}
		\right\}.
		\]

		For any \(\lambda\in[0,\lambda_0]\), recursively applying
		\eqref{RS_equ_HP_MGF_recursive_bound} yields, for
		\(t=0,\ldots,T-1\),
		\[
		\begin{aligned}
			\mathbb E
			\left[
			\exp(\lambda\widetilde V_t)
			\right]
			&\le
			\exp
			\left(
			\lambda\alpha^t\Delta_{y,0}^2
			+
			\lambda(\nu_t+\kappa_t)
			\sum_{i=0}^{t-1}\alpha^i
			\right)
			\\
			&\le
			\exp
			\left(
			\lambda
			\left[
			\alpha^t\Delta_{y,0}^2
			+
			\left(
			\frac{20\eta_y}{\mu}
			\left(
			1+\frac{\mu\eta_y}{2}
			\right)
			\sigma_y^2
			+
			\frac{2}{\mu\eta_y}
			\left(
			1+\frac{2}{\mu\eta_y}
			\right)
			\kappa^2\eta_x^2
			\right)
			\right]
			\right),
		\end{aligned}
		\]
		where \(\Delta_{y,0}:=\|y^0-y^{*}(x^0)\|\).
		Taking \(\lambda=\lambda_0\) in the preceding inequality and applying Markov's inequality, we obtain, for each
		fixed \(t=0,\ldots,T-1\), with probability at least
		\(1-\delta/(2T)\),
		\begin{equation}\label{RS_equ_HP_ytracking_bound}
			\widetilde V_t
			\le
			\alpha^t\Delta_{y,0}^2
			+\Gamma_y,
		\end{equation}
		where 
		\[\Gamma_y:=
		\frac{20\eta_y}{\mu}
		\left(
		1+\frac{\mu\eta_y}{2}
		\right)
		\sigma_y^2
		+
		\frac{2}{\mu\eta_y}
		\left(
		1+\frac{2}{\mu\eta_y}
		\right)
		\kappa^2\eta_x^2
		+
		\frac{1}{\lambda_0}
		\log\frac{2T}{\delta}.
		\]
		Then, 
		\[
		\mathbb P(\mathcal E_y)
		\geq
		1-\frac{\delta}{2},
		\]
		where 
		\[
		\mathcal E_y
		:=
		\bigcap_{t=0}^{T-1}
		\left\{
		\widetilde V_t
		\le \alpha^t\Delta_{y,0}^2+\Gamma_y
		\right\}.
		\]

		In what follows, by contradiction, we prove that \(\mathcal E_y\) implies \(\tau\ge T\). Suppose that \(\tau<T\); then
		\(\tau\in\{0,\ldots,T-1\}\), and the preceding bound \eqref{RS_equ_HP_ytracking_bound} holds at \(t=\tau\). This implies
		\[
		\begin{aligned}
			\|y^\tau-y^*(x^\tau)\|^2
			&\leq
			\alpha^{\tau}\Delta_{y,0}^2
			+
			\Gamma_y.
		\end{aligned}
		\]

		For \(\Gamma_y\), by the parameter choices in the theorem, 
		\begin{equation*}
			\begin{aligned}
				\Gamma_y\leq{}&\left(\frac{22}{\log4}+242\right)
				\left(
				\frac{\ell_T}{T}
				\right)^{1/2}
				\frac{\sigma_y^2}{\mu^2}
				\left(
				\frac{4\kappa^2\mu^2}
				{121\sigma_y^2(1+\kappa)^2L_{x,0}^2}
				\right)^{1/3}\\
				&+\frac{2}{\mu\eta_y}
				\left(
				1+\frac{2}{\mu\eta_y}
				\right)
				\kappa^2\eta_x^2,
			\end{aligned}
		\end{equation*}
		where the inequality uses \(\mu\eta_y\leq\frac15\) and \(\ell_T\geq\log4\). For the second term on the right-hand side of the above inequality, substituting the choices
		of \(\eta_x\) and \(\eta_y\) yields
		\[
		\begin{aligned}
			\frac{2}{\mu\eta_y}
			\left(
			1+\frac{2}{\mu\eta_y}
			\right)
			\kappa^2\eta_x^2
			={}&
			\left(
			\frac{\ell_T}{T}
			\right)^{1/2}
			\frac{2(1-\beta)\kappa^2}
			{(1+\kappa)^2L_{x,0}^2}
			\left(
			\frac{121\sigma_y^2(1+\kappa)^2L_{x,0}^2}
			{4\kappa^2\mu^2}
			\right)^{1/3}
			\\
			&+
			\left(
			\frac{\ell_T}{T}
			\right)^{1/2}
			\frac{4\kappa^2}
			{(1+\kappa)^2L_{x,0}^2}
			\left(
			\frac{121\sigma_y^2(1+\kappa)^2L_{x,0}^2}
			{4\kappa^2\mu^2}
			\right)^{2/3}.							
		\end{aligned}
		\]
		Moreover, by the definition of \(\eta_y\) and the condition
		\(\eta_y \leq 1/(5\mu)\),
		\[
		1-\beta
		\leq
		\frac15
		\left(
		\frac{121\sigma_y^2(1+\kappa)^2L_{x,0}^2}
		{4\kappa^2\mu^2}
		\right)^{1/3}.
		\]
		Then,					
		\[
		\begin{aligned}
			\frac{2}{\mu\eta_y}
			\left(
			1+\frac{2}{\mu\eta_y}
			\right)
			\kappa^2\eta_x^2
			\leq{}&
			\left(
			\frac{\ell_T}{T}
			\right)^{1/2}
			\frac{22\kappa^2}
			{5(1+\kappa)^2L_{x,0}^2}
			\left(
			\frac{121\sigma_y^2(1+\kappa)^2L_{x,0}^2}
			{4\kappa^2\mu^2}
			\right)^{2/3}
			\\
			={}&
			\frac{1331}{10}
			\left(
			\frac{\ell_T}{T}
			\right)^{1/2}
			\frac{\sigma_y^2}{\mu^2}
			\left(
			\frac{4\kappa^2\mu^2}
			{121\sigma_y^2(1+\kappa)^2L_{x,0}^2}
			\right)^{1/3}.
		\end{aligned}
		\]
		Subsequently,
		\[
		\begin{aligned}
			\Gamma_y
			\leq{}&
			\left(
			\frac{3751}{10}
			+
			\frac{22}{\log4}
			\right)
			\left(
			\frac{\ell_T}{T}
			\right)^{1/2}
			\frac{\sigma_y^2}{\mu^2}
			\left(
			\frac{4\kappa^2\mu^2}
			{121\sigma_y^2(1+\kappa)^2L_{x,0}^2}
			\right)^{1/3}
			\\
			={}&
			\left(
			\frac{3751}{10}
			+
			\frac{22}{\log4}
			\right)
			\frac{\eta_y\sigma_y^2}{\mu}
			\ell_T\\
			\leq{}&
			\frac{15}{256L_{x,1}^2},
		\end{aligned}
		\]
		where the equality uses
		\(
		\left(
		\frac{\ell_T}{T}
		\right)^{1/2}
		=
		(1-\beta)\ell_T
		\)
		and the definition of \(\eta_y\), and the second inequality follows from the upper bound on \(\eta_y\).
		
		It follows that
		\[
		\begin{aligned}
			\|y^\tau-y^*(x^\tau)\|^2
			&\leq
			\alpha^\tau\Delta_{y,0}^2
			+
			\Gamma_y
			\\
			&\leq
			\frac{1}{256L_{x,1}^2}
			+
			\frac{15}{256L_{x,1}^2}
			=
			\frac{1}{16L_{x,1}^2},
		\end{aligned}
		\]
		where the second inequality uses \(\alpha\in (0,1)\) and
		\(\Delta_{y,0}\le1/(16L_{x,1})\).
		Hence,
		\[
		\|y^\tau-y^*(x^\tau)\|
		\leq
		\frac{1}{4L_{x,1}},
		\]
		which contradicts the definition of
		\(\tau\).

		Define
		\[
		\mathcal E:=\mathcal E_x\cap\mathcal E_y,
		\]
		then
		\[
		\mathbb P(\mathcal E)\geq 1-\delta.
		\]
		We proceed to bound \(I_2\) when the event \(\mathcal E\) holds.
		Since \(\mathcal E_y\) implies \(\tau\ge T\), we have by Assumption~\ref{RS_ass_L0L1Smooth} that
		\begin{equation*}
			\begin{aligned}
				I_2\leq{}&\underbrace{
					(1-\beta)L_{x,0}
					\sum_{t=0}^{T-1}
					\sum_{i=0}^{t}
					\beta^{t-i}
					\|y^i-y^*(x^i)\|}_{=:I_3}\\
				&+
				\underbrace{
					(1-\beta)L_{x,1}
					\sum_{t=0}^{T-1}
					\sum_{i=0}^{t}
					\beta^{t-i}
					\|\nabla\Phi(x^i)\|
					\|y^i-y^*(x^i)\|}_{=:I_4}
			\end{aligned}
		\end{equation*}

		For \(I_3\), since \(\tau\geq T\), we have
		\[
		\widetilde V_i
		=
		\|y^i-y^*(x^i)\|^2,
		\qquad
		i=0,\ldots,T-1,
		\]
		and the definition of \(\mathcal E_y\) implies
		\[
		\|y^i-y^*(x^i)\|
		\leq
		\alpha^{i/2}\Delta_{y,0}
		+
		\sqrt{\Gamma_y},
		\qquad
		i=0,\ldots,T-1.
		\]
		Then,
		\[
		\begin{aligned}
			I_3
			&\leq
			(1-\beta)L_{x,0}
			\sum_{t=0}^{T-1}
			\sum_{i=0}^{t}
			\beta^{t-i}
			\left(
			\alpha^{i/2}\Delta_{y,0}
			+
			\sqrt{\Gamma_y}
			\right)
			\\
			&\leq
			L_{x,0}\Delta_{y,0}
			\sum_{i=0}^{T-1}\alpha^{i/2}
			+
			TL_{x,0}\sqrt{\Gamma_y} \\
			&\leq
			\frac{4L_{x,0}\Delta_{y,0}}
			{\mu\eta_y}
			+
			TL_{x,0}\sqrt{\Gamma_y},
		\end{aligned}
		\]
		where the last inequality follows from 
		\[
		\begin{aligned}
			\sum_{i=0}^{T-1}\alpha^{i/2}
			\leq
			\frac{1}{1-\sqrt{\alpha}}
			\leq
			\frac{2}{1-\alpha}
			=
			\frac{4}{\mu\eta_y}.
		\end{aligned}
		\]

		For \(I_4\), 
		\[
		\begin{aligned}
			I_4
			&\leq \frac{1}{4}(1-\beta)
			\sum_{t=0}^{T-1}
			\sum_{i=0}^{t}
			\beta^{t-i}
			\|\nabla\Phi(x^i)\|\\
			&=
			\frac{1}{4}
			\sum_{i=0}^{T-1}
			\left(
			(1-\beta)
			\sum_{t=i}^{T-1}
			\beta^{t-i}
			\right)
			\|\nabla\Phi(x^i)\|
			\\
			&\leq
			\frac{1}{4}
			\sum_{t=0}^{T-1}
			\|\nabla\Phi(x^t)\|,
		\end{aligned}
		\]
		where the first inequality uses
		\(\|y^i-y^*(x^i)\|\leq 1/(4L_{x,1})\) for
		\(i=0,\ldots,T-1\), and the last inequality uses \(\beta\in (0,1)\).

		Subsequently,  when the event  \(\mathcal E\) holds, \eqref{RS_equ_HP_momentum_error_bound} gives	
		\begin{equation}
			\label{RS_equ_HP_momentum_sum}
			\begin{aligned}
				\sum_{t=0}^{T-1}
				\|\varepsilon_t\|
				\leq{}&
				4T\sigma_x
				\sqrt{
					\frac{1-\beta}{1+\beta}
					\log\frac{4}{\delta}
				}
				+
				\frac{4L_{x,0}\Delta_{y,0}}
				{\mu\eta_y}
				\\
				&+
				TL_{x,0}\sqrt{\Gamma_y}
				+
				\left(
				\frac14
				+
				\frac{\eta_xL_1\beta}{1-\beta}
				\right)
				\sum_{t=0}^{T-1}
				\|\nabla\Phi(x^t)\|
				\\
				&+
				\frac{T\eta_xL_0\beta}{1-\beta}
				+
				\frac{\beta}{1-\beta}
				\|m^0-\nabla\Phi(x^0)\|.
			\end{aligned}
		\end{equation}
		
		Substituting \eqref{RS_equ_HP_momentum_sum} into
		\eqref{RS_equ_HP_descent} yields
		\begin{equation*}
			\begin{aligned}
				C_0
				\frac{1}{T}
				\sum_{t=0}^{T-1}
				\|\nabla\Phi(x^t)\|
				\leq{}&
				\frac{\Phi(x^0)-\Phi^*}{T\eta_x}
				+
				\frac{8L_{x,0}\Delta_{y,0}}
				{T\mu\eta_y}
				+
				2L_{x,0}\sqrt{\Gamma_y}
				\\
				&+
				8\sigma_x
				\sqrt{
					(1-\beta)
					\log\frac{4}{\delta}
				}
				\\
				&+
				\frac{2\beta}
				{T(1-\beta)}
				\|m^0-\nabla\Phi(x^0)\|
				\\
				&+
				\frac{\eta_xL_0}{2}
				+
				\frac{2\eta_xL_0\beta}
				{1-\beta},
			\end{aligned}
		\end{equation*}
		where
		\[
		C_0
		:=
		1-\frac{\eta_xL_1}{2}
		-\frac12
		-\frac{2\eta_xL_1\beta}{1-\beta}.
		\]
		
		Since \(\eta_x \leq \frac{1-\beta}{8L_1}\), it follows that
		\(C_0 \geq \frac14\).
		Substituting the parameter choices of the theorem into the above bound, we obtain
		\begin{equation*}
			\begin{aligned}
				\frac{1}{T}
				\sum_{t=0}^{T-1}
				\|\nabla\Phi(x^t)\|
				\leq{}& 4(1+\kappa)L_{x,0}\left(\Phi(x^0)-\Phi^*\right)
				\left(
				\frac{\ell_T}{T}
				\right)^{1/4}\\
				&+32L_{x,0}\Delta_{y,0}\left(
				\frac{121\sigma_y^2(1+\kappa)^2L_{x,0}^2}
				{4\kappa^2\mu^2}
				\right)^{1/3}
				\left(
				\frac{\ell_T}{T}
				\right)^{1/2}\\
				&+\frac{8L_{x,0}\sigma_y}{\mu}
				\Bigg[
				\left(
				\frac{3751}{10}+\frac{22}{\log4}
				\right)
				\left(
				\frac{4\kappa^2\mu^2}
				{121\sigma_y^2(1+\kappa)^2L_{x,0}^2}
				\right)^{1/3}
				\Bigg]^{1/2}
				\left(
				\frac{\ell_T}{T}
				\right)^{1/4}\\
				&+\left(32\sigma_x+8\right)\left(
				\frac{\ell_T}{T}
				\right)^{1/4}+8\|m^0-\nabla\Phi(x^0)\|\left(
				\frac{\ell_T}{T}
				\right)^{1/2}+2\left(\frac{\ell_T}{T}\right)^{3/4}.
			\end{aligned}
		\end{equation*}
		Then there exists a constant \(C>0\), independent of \(T\) and
		\(\delta\), such that
		\[
		\frac1T
		\sum_{t=0}^{T-1}
		\|\nabla\Phi(x^t)\|
		\leq
		C
		\left(
		\frac{\ell_T}{T}
		\right)^{1/4}
		\]
		holds with probability at least \(1-\delta\).
		The proof is complete.
	\end{proof}
	
	Theorem~\ref{RS_thm_NSGDA-M_highProba} shows that NSGDA-M requires \(\mathcal O\left(\epsilon^{-4}\log(e/(\delta\epsilon))\right)\) iterations to achieve an \(\epsilon\)-stationary point of the primal function with probability at least \(1-\delta\), with
	\[
	\beta=1-\Theta\left(\epsilon^2/\log(e/(\delta\epsilon))\right),\quad
	\eta_y=\Theta\left(\epsilon^2/\log(e/(\delta\epsilon))\right),\quad
	\eta_x=\Theta\left(\epsilon^3/\log(e/(\delta\epsilon))\right).
	\]
	Compared with the high-probability complexity
	\(\mathcal O(\delta^{-4}\epsilon^{-4})\) established in
	Xian et al.~\cite{xian2024delving}, NSGDA-M improves the dependence on
	the failure probability \(\delta\) from polynomial to logarithmic with
	constant batch size.
	This improvement results from a trajectory-wise concentration analysis of
	the martingale-difference noise terms, rather than by converting
	localized expected stationarity bounds into probability guarantees via
	Markov's inequality.
	
	Similar to Xian et al.~\cite{xian2024delving}, our convergence analysis requires a warm-start condition on the initial
	point \(y^0\) to ensure that the dual tracking error remains within the
	localization region where the generalized smoothness condition holds.
	For a given \(x^0\), this condition amounts to finding an approximate
	solution of the initial inner maximization problem
	\[
	\max_{y\in\mathcal Y}\mathcal L(x^0,y),
	\]
	which can be obtained by applying existing first-order methods for
	generalized-smooth strongly convex optimization; see, e.g.,
	\cite{xian2024delving,chen2023generalized,li2023convex,fang2018spider}.

	\subsection{Convergence analysis of NSGDA-M in expectation}\label{RS_Subsection_3_2}
	This subsection studies the convergence of Algorithm~\ref{algorithm:NSGDA-M} in expectation. Unlike the preceding high-probability analysis, the expectation analysis needs to control the effect of the localization failure event to derive an unconditional expectation bound. 
	This requires the following bounded-gradient condition along the iterates.
	
	\begin{ass}
		\label{RS_ass_boundedGradientIterates}
		There exists a constant \(G_\Phi>0\) such that the iterates
		generated by Algorithm~\ref{algorithm:NSGDA-M} satisfy
		\[
		\|\nabla\Phi(x^t)\|
		\leq
		G_\Phi,
		\qquad
		t\geq0,
		\quad
		\text{a.s.}
		\]
	\end{ass}

	\begin{ass}
		\label{RS_ass_varianceBound}
		There exist constants \(\sigma_x,\sigma_y>0\) such that, for all
		\(t\geq0\),
		\[
		\begin{aligned}
			&\mathbb E
			\left[
			\|G_x(x^t,y^t,\xi^t)
			-
			\nabla_x\mathcal L(x^t,y^t)\|^2
			\mid
			\mathcal F_t
			\right]
			\leq
			\sigma_x^2,\\
			&\mathbb E
			\left[
			\|G_y(x^t,y^t,\xi^t)
			-
			\nabla_y\mathcal L(x^t,y^t)\|^2
			\mid
			\mathcal F_t
			\right]
			\leq
			\sigma_y^2 .
		\end{aligned}
		\]
	\end{ass}
	Assumption~\ref{RS_ass_boundedGradientIterates} is imposed solely to
	control the effect of the localization failure event when deriving
	the unconditional expectation bound. Assumption~\ref{RS_ass_varianceBound} is the standard bounded-variance condition for stochastic gradient estimators.

	The following theorem establishes an expected stationarity guarantee
	for NSGDA-M.
	
	\begin{theorem}
		\label{RS_thm_NSGDA-M_expectation}
		Let \(L_{x,0},L_{x,1},L_{y,0},L_{y,1}\) denote the generalized
		smoothness constants, \(\sigma_x,\sigma_y\) denote the noise parameters,
		and \(T\) denote the iteration horizon. Choose
		\[
		1-\beta
		=
		\frac{\mu\sqrt{\delta}}
		{40L_{x,1}\sigma_y\sqrt{T}},
		\qquad
		\eta_y
		=
		\frac{5(1-\beta)}{\mu},
		\qquad
		\eta_x
		=
		\frac{\sqrt{5(1-\beta)\delta}}
		{16\kappa L_{x,1}\sqrt{T}},
		\]
		where \(\delta\in(0,1)\) and $L_y$ and $\kappa$ are defined in
		\eqref{RS_equ:conditionNum_def}. Suppose that \(T\) is sufficiently
		large such that
		\[
		1-\beta\leq\frac13,
		\qquad
		\eta_y\leq\frac{\mu}{L_y^2},
		\qquad
		\eta_x
		\leq
		\frac{1-\beta}{8(1+\kappa)L_{x,1}},
		\]
		and the initial tracking error satisfies
		\(
		\|y^0-y^*(x^0)\|
		\leq
		\frac{\sqrt{\delta}}{16L_{x,1}},
		\)
		and Assumptions
		\ref{RS_ass_PhiLowerBound}--\ref{RR_ass:unbiased}
		and
		\ref{RS_ass_boundedGradientIterates}--\ref{RS_ass_varianceBound}
		hold. Then there exists a constant \(C>0\), independent of
		\(T\) and \(\delta\), such that
		\begin{equation}\label{RS_equ_NSGDA-M_expectation_bound}
			\frac1T
			\sum_{t=0}^{T-1}
			\mathbb E
			\left[
			\|\nabla\Phi(x^t)\|
			\right]
			\leq
			\frac{C}{\delta^{3/4}T^{1/4}}
			+
			G_\Phi\delta ,
		\end{equation}
		which further implies
		\[
		\frac1T
		\sum_{t=0}^{T-1}
		\mathbb E
		\left[
		\|\nabla\Phi(x^t)\|
		\right]
		\leq
		C T^{-1/7}.
		\]
	\end{theorem}

	\begin{proof}
		Recall the stopping time
		\[
		\tau
		:=
		\inf
		\left\{
		t\geq0:
		\|y^t-y^{*}(x^t)\|
		>
		\frac{1}{4L_{x,1}}
		\right\},\quad\quad \inf\emptyset=+\infty.
		\]
		Then, 
		\begin{equation}\label{RS_equ_stopped_expectation_split}
			\begin{aligned}
				\frac1T
				\sum_{t=0}^{T-1}
				\mathbb E\|\nabla\Phi(x^t)\|
				={}&
				\frac1T
				\sum_{t=0}^{T-1}
				\mathbb E
				\left[
				\|\nabla\Phi(x^t)\|
				\mathbf 1_{\{\tau<T\}}
				\right]
				+
				\frac1T
				\sum_{t=0}^{T-1}
				\mathbb E
				\left[
				\|\nabla\Phi(x^t)\|
				\mathbf 1_{\{\tau\geq T\}}
				\right]\\
				\leq{}& \frac1T
				\sum_{t=0}^{T-1}
				\mathbb E
				\left[
				\|\nabla\Phi(x^t)\|
				\mathbf 1_{\{\tau<T\}}\right]
				+\frac1T
				\sum_{t=0}^{T-1}
				\mathbb E
				\left[
				\|\nabla\Phi(x^t)\|
				\mathbf 1_{\{t< \tau\}}
				\right],
			\end{aligned}
		\end{equation}
		where the last inequality follows from
		\(\{\tau\geq T\}\subseteq\{t<\tau\}\) for every
		\(t=0,\ldots,T-1\).
		
		In what follows, we first bound the first term on the right-hand side of \eqref{RS_equ_stopped_expectation_split} by Markov's inequality.
		For notational simplicity, we define \(e_t:=y^t-y^*(x^t)\), then
		\[
		\mathbb P(\tau<T)
		=
		\mathbb P
		\left(
		\sum_{j=0}^{T-1}
		\|e_j\|^2
		\mathbf 1_{\{\tau=j\}}
		>
		\frac{1}{16L_{x,1}^2}
		\right).
		\]	
		To this end, we show the following inequality holds by induction
		\begin{equation}\label{eq:stopped_moment}
			\begin{aligned}
				\mathbb E\|e_{t\wedge\tau}\|^2
				={}&
				\mathbb E
				\left[
				\|e_t\|^2\mathbf 1_{\{t\leq\tau\}}
				\right]
				+
				\sum_{j=0}^{t-1}
				\mathbb E
				\left[
				\|e_j\|^2\mathbf 1_{\{\tau=j\}}
				\right]\\                                             
				\leq{}&
				\|e_0\|^2+tC_1, \quad t=0,1,\ldots,T.
			\end{aligned}
		\end{equation}
		
		Obviously, \eqref{eq:stopped_moment} holds for \(t=0\).
		Assume that \eqref{eq:stopped_moment} holds for some
		\(t\in\{0,\ldots,T-1\}\).
		By the definition of \(e_{t\wedge\tau}\),
		\begin{equation}\label{RS_equ_stopped_moment_induction_step}
			\begin{aligned}
				\mathbb E\|e_{(t+1)\wedge\tau}\|^2={}&
				\mathbb E
				\left[
				\|e_{t+1}\|^2\mathbf 1_{\{t+1\leq\tau\}}
				\right]
				+
				\sum_{j=0}^{t}
				\mathbb E
				\left[
				\|e_j\|^2\mathbf 1_{\{\tau=j\}}
				\right]                                                \\
				={}&
				\mathbb E
				\left[
				\mathbf 1_{\{t<\tau\}}
				\mathbb E
				\left[
				\|e_{t+1}\|^2
				\,\middle|\,
				\mathcal F_t
				\right]
				\right]
				+
				\sum_{j=0}^{t}
				\mathbb E
				\left[
				\|e_j\|^2\mathbf 1_{\{\tau=j\}}
				\right].
			\end{aligned}
		\end{equation}
		
		By Young's inequality, for any \(\theta>0\),
		\begin{equation}\label{RS_equ:sto_Young}
			\begin{aligned}
				&\mathbb E
				\left[
				\|e_{t+1}\|^2
				\,\middle|\,
				\mathcal F_t
				\right]\\
				\leq&
				\left(
				1+\frac{1}{\theta}
				\right)
				\mathbb E
				\left[
				\|y^*(x^{t+1})-y^*(x^t)\|^2
				\,\middle|\,
				\mathcal F_t
				\right]
				+
				(1+\theta)
				\mathbb E
				\left[
				\|y^{t+1}-y^*(x^t)\|^2
				\,\middle|\,
				\mathcal F_t
				\right]\\
				\leq&\left(1+\frac{1}{\theta}\right)\kappa^2\eta_x^2+(1+\theta)
				\mathbb E\left[\|y^{t+1}-y^*(x^t)\|^2\,\middle|\,\mathcal F_t\right],
			\end{aligned}
		\end{equation}
		where the second inequality follows from Lemma~\ref{RS_lem_LipSolution}, the update rule of \(x^t\) and \(\eta_x\leq \frac{1}{L_{x,1}}\).

		For the last term of \eqref{RS_equ:sto_Young}, we have by the update rule for \(y\) and the optimality of \(y^*(x^t)\) that
		\begin{equation*}
			\begin{aligned}
				&\mathbb{E}\left[\|y^{t+1}-y^*(x^t)\|^2\,\middle|\,\mathcal F_t\right] \\                                      
				={}&
				\mathbb{E}\left[\left\|
				\operatorname{proj}_{\mathcal Y}
				\left(
				y^t+\eta_yG_y(x^t,y^t,\xi^t)
				\right)
				-
				\operatorname{proj}_{\mathcal Y}
				\left(
				y^*(x^t)
				+
				\eta_y\nabla_y\mathcal L(x^t,y^*(x^t))
				\right)
				\right\|^2\,\middle|\,\mathcal F_t\right]                                             \\
				\leq{}&
				\mathbb{E}\left[\left\|
				y^t-y^*(x^t)
				+
				\eta_y
				\left(
				G_y(x^t,y^t,\xi^t)
				-
				\nabla_y\mathcal L(x^t,y^*(x^t))
				\right)
				\right\|^2\,\middle|\,\mathcal F_t\right]                                            \\
				\leq{}&
				\|e_t\|^2                                    
				+
				2\eta_y
				\left\langle
				\nabla_y\mathcal L(x^t,y^t)
				-
				\nabla_y\mathcal L(x^t,y^*(x^t)),
				y^t-y^*(x^t)
				\right\rangle                                         \\
				&+
				\eta_y^2\sigma_y^2                                    
				+
				\eta_y^2
				\left\|
				\nabla_y\mathcal L(x^t,y^t)
				-
				\nabla_y\mathcal L(x^t,y^*(x^t))
				\right\|^2\\
				\leq{}&
				\left(
				1-2\mu\eta_y+\eta_y^2L_y^2
				\right)
				\|e_t\|^2
				+
				\eta_y^2\sigma_y^2\\
				\leq{}&
				(1-\mu\eta_y)\|e_t\|^2
				+
				\eta_y^2\sigma_y^2,
			\end{aligned}
		\end{equation*}
		where the first inequality follows from the nonexpansiveness of the projection operator, the second inequality follows from  Assumptions~\ref{RR_ass:unbiased} and \ref{RS_ass_varianceBound},  the third inequality follows from the \(\mu\)-strong concavity of
		\(\mathcal L(x^t,\cdot)\) and Assumptions~\ref{RS_ass_L0L1Smooth}
		and~\ref{RS_ass_gradyStarBound}, and the last inequality follows from \(\eta_y\leq\mu/L_y^2\).
		
		Then, by setting \(\theta=\frac{\mu\eta_y}{2}\) in \eqref{RS_equ:sto_Young}, we have, for \(t<\tau\),
		\begin{equation}\label{eq:onestep_stopped}
			\mathbb E
			\left[
			\|e_{t+1}\|^2
			\,\middle|\,
			\mathcal F_t
			\right]
			\leq
			\gamma\|e_t\|^2+C_1,
		\end{equation}
		where
		\[
		\gamma:=1-\frac{\mu\eta_y}{2},
		\qquad
		C_1
		:=
		\left(
		1+\frac{2}{\mu\eta_y}
		\right)
		\kappa^2\eta_x^2
		+
		\left(
		1+\frac{\mu\eta_y}{2}
		\right)
		\eta_y^2\sigma_y^2 .
		\]
		
		Substituting \eqref{eq:onestep_stopped} into \eqref{RS_equ_stopped_moment_induction_step} yields
		\[
		\begin{aligned}
			\mathbb E\|e_{(t+1)\wedge\tau}\|^2={}&
			\mathbb E
			\left[
			\|e_{t+1}\|^2\mathbf 1_{\{t+1\leq\tau\}}
			\right]
			+
			\sum_{j=0}^{t}
			\mathbb E
			\left[
			\|e_j\|^2\mathbf 1_{\{\tau=j\}}
			\right]                                                \\
			\leq{}&
			\gamma
			\mathbb E
			\left[
			\|e_t\|^2\mathbf 1_{\{t<\tau\}}
			\right]
			+
			C_1\mathbb P(t<\tau)
			+
			\sum_{j=0}^{t}
			\mathbb E
			\left[
			\|e_j\|^2\mathbf 1_{\{\tau=j\}}
			\right]\\
			\leq{}&
			\mathbb E
			\left[
			\|e_t\|^2\mathbf 1_{\{t\leq\tau\}}
			\right]
			+
			\sum_{j=0}^{t-1}
			\mathbb E
			\left[
			\|e_j\|^2\mathbf 1_{\{\tau=j\}}
			\right]
			+
			C_1 \\
			\leq{}&
			\|e_0\|^2+(t+1)C_1,
		\end{aligned}
		\]
		where the second inequality follows from \(\gamma\in(0,1)\) and the disjoint decomposition 
		\(
		\{t\leq\tau\}
		=
		\{t<\tau\}\cup\{\tau=t\},
		\)
		and the last inequality follows from the induction hypothesis. Thus \eqref{eq:stopped_moment} holds.	
		
		Taking \(t=T\) in \eqref{eq:stopped_moment} gives
		\begin{equation}\label{eq:stopped_moment_T}
			\mathbb E
			\left[
			\|e_T\|^2\mathbf 1_{\{T\leq\tau\}}
			\right]
			+
			\sum_{j=0}^{T-1}
			\mathbb E
			\left[
			\|e_j\|^2\mathbf 1_{\{\tau=j\}}
			\right]
			\leq
			\|e_0\|^2+TC_1 .
		\end{equation}

		When the event \(\{\tau<T\}\) holds, we have by the definition of
		\(\tau\) that
		\[
		\sum_{j=0}^{T-1}
		\|e_j\|^2\mathbf 1_{\{\tau=j\}}
		=
		\|e_\tau\|^2
		>
		\frac{1}{16L_{x,1}^2}.
		\]
		Therefore, by Markov's inequality and
		\eqref{eq:stopped_moment_T},
		\begin{equation}\label{eq:prob_bound}
			\begin{aligned}
				\mathbb P(\tau<T)
				&=\mathbb P
				\left(
				\sum_{j=0}^{T-1}
				\|e_j\|^2\mathbf 1_{\{\tau=j\}}>
				\frac{1}{16L_{x,1}^2}
				\right)                                             \\
				&\leq
				16L_{x,1}^2
				\left(
				\|e_0\|^2+TC_1
				\right)\\
				&\leq \frac{\delta}{16}+\frac{11\delta}{16}\leq
				\delta,
			\end{aligned}
		\end{equation}
		where the second inequality follows from
		\(\|e_0\|\leq\frac{\sqrt{\delta}}{16L_{x,1}}\), the choice of \(\eta_x\) and \(\eta_y\) in the theorem, and \(1-\beta\leq\frac{1}{3}\).
		Hence, by Assumption~\ref{RS_ass_boundedGradientIterates},
		\begin{equation}\label{RS_equ_stopped_localization_failure}
			\begin{aligned}
				\frac1T
				\sum_{t=0}^{T-1}
				\mathbb E
				\left[
				\|\nabla\Phi(x^t)\|
				\mathbf 1_{\{\tau<T\}}
				\right]
				&\leq
				G_\Phi\mathbb P(\tau<T)
				\\
				&\leq
				G_\Phi\delta.
			\end{aligned}
		\end{equation}

		Next, we bound the second term on the right-hand side of \eqref{RS_equ_stopped_expectation_split}.
		Since \(\Phi(\cdot)\) is \((L_0,L_1)\)-smooth with \(L_0:=(1+\kappa)L_{x,0}\) and \(L_1:=(1+\kappa)L_{x,1}\), we have by the update rule of \(x\) in \eqref{RS_equ:alg_x} and \(\eta_x\leq \frac{1}{L_1}\) that
		\[
		\begin{aligned}
			\Phi(x^{t+1})
			\leq{}&
			\Phi(x^t)
			+
			\left\langle
			\nabla\Phi(x^t),
			x^{t+1}-x^t
			\right\rangle
			+
			\frac{L_0+L_1\|\nabla\Phi(x^t)\|}{2}
			\|x^{t+1}-x^t\|^2
			\\
			\leq{}&
			\Phi(x^t)
			-
			\eta_x
			\left\langle
			m^{t+1}-\varepsilon_t,
			\frac{m^{t+1}}{\|m^{t+1}\|}
			\right\rangle
			+
			\frac{\eta_x^2}{2}
			\left(
			L_0+L_1\|\nabla\Phi(x^t)\|
			\right)
			\\
			\leq{}&
			\Phi(x^t)
			-
			\eta_x\|m^{t+1}\|
			+
			\eta_x\|\varepsilon_t\|
			+
			\frac{\eta_x^2}{2}
			\left(
			L_0+L_1\|\nabla\Phi(x^t)\|
			\right)
			\\
			\leq{}&
			\Phi(x^t)
			-
			\eta_x\|\nabla\Phi(x^t)\|
			+
			2\eta_x\|\varepsilon_t\|
			+
			\frac{\eta_x^2}{2}
			\left(
			L_0+L_1\|\nabla\Phi(x^t)\|
			\right),
		\end{aligned}
		\]
		where \(\varepsilon_t:=m^{t+1}-\nabla\Phi(x^t)\), the second inequality follows from
		Cauchy--Schwarz inequality, and the last inequality follows from the fact
		\[
		\|m^{t+1}\|
		=
		\|\nabla\Phi(x^t)+\varepsilon_t\|
		\geq
		\|\nabla\Phi(x^t)\|
		-
		\|\varepsilon_t\|.
		\]
		Multiplying the above inequality by
		\(\mathbf{1}_{\{t<\tau\}}\), summing over
		\(t=0,\ldots,T-1\), and taking expectations gives
		\begin{equation}\label{eq:stopped_descent}
			\begin{aligned}
				&
				\left(
				1-\frac{\eta_xL_1}{2}
				\right)
				\frac1T
				\sum_{t=0}^{T-1}
				\mathbb{E}
				\left[
				\|\nabla\Phi(x^t)\|
				\mathbf{1}_{\{t<\tau\}}
				\right]                                            \\
				\leq{}&
				\frac2T
				\sum_{t=0}^{T-1}
				\mathbb{E}
				\left[
				\|\varepsilon_t\|
				\mathbf{1}_{\{t<\tau\}}
				\right]
				+
				\frac{\Phi(x^0)-\Phi^*}{T\eta_x}
				+
				\frac{\eta_xL_0}{2}.
			\end{aligned}
		\end{equation}

		For the first term on the right-hand side of \eqref{eq:stopped_descent}, 
		we have by \eqref{eq:momentum_error_bound} in Theorem~\ref{RS_thm_NSGDA-M_highProba} that
		\[
		\begin{aligned}
			\|\varepsilon_t\|
			\leq{}&
			\beta^{t+1}
			\|m^0-\nabla\Phi(x^0)\|                              
			+
			\eta_x\beta
			\sum_{i=0}^{t-1}
			\beta^{t-i-1}
			\left(
			L_0+L_1\|\nabla\Phi(x^i)\|
			\right)                                                \\
			&+
			(1-\beta)
			\left\|
			\sum_{i=0}^{t}
			\beta^{t-i}
			\widehat{\varepsilon}_i
			\right\|,
		\end{aligned}
		\]
		where 
		\(
		\widehat{\varepsilon}_i
		\)
		is defined in \eqref{eq:gradient_estimator_error}.	
		Then,
		\[
		\begin{aligned}
			\sum_{t=0}^{T-1}
			\mathbb E
			\left[
			\|\varepsilon_t\|\mathbf 1_{\{t<\tau\}}
			\right]                                                
			\leq{}&
			(1-\beta)
			\sum_{t=0}^{T-1}
			\mathbb E
			\left[
			\left\|
			\sum_{i=0}^{t}
			\beta^{t-i}\widehat{\varepsilon}_i
			\right\|
			\mathbf 1_{\{t<\tau\}}
			\right]  \\                                              
			&+
			\frac{\eta_xL_1\beta}{1-\beta}
			\sum_{t=0}^{T-1}
			\mathbb E
			\left[
			\|\nabla\Phi(x^t)\|
			\mathbf 1_{\{t<\tau\}}
			\right]   \\                                             
			&+
			\frac{T\eta_xL_0\beta}{1-\beta}
			+
			\frac{\beta}{1-\beta}
			\|m^0-\nabla\Phi(x^0)\|.
		\end{aligned}
		\]
		Denote
		\(
		\Delta_x^i
		:=
		G_x(x^i,y^i,\xi^i)
		-
		\nabla_x\mathcal L(x^i,y^i).
		\)
		We have by the triangle inequality and Assumption \ref{RS_ass_L0L1Smooth} that
		\begin{equation}\label{RS_equ_HP_I1tau_decomposition}
			\begin{aligned}
				\sum_{t=0}^{T-1}
				\mathbb E
				\left[
				\|\varepsilon_t\|\mathbf 1_{\{t<\tau\}}
				\right]                                                                     
				\leq{}&
				\underbrace{
					(1-\beta)
					\sum_{t=0}^{T-1}
					\mathbb E
					\left[
					\left\|
					\sum_{i=0}^{t}
					\beta^{t-i}
					\Delta_x^i
					\right\|
					\mathbf 1_{\{t<\tau\}}
					\right]}_{=:I_{1,\tau}}+
				\frac{T\eta_xL_0\beta}{1-\beta}                                \\
				&+
				\underbrace{
					(1-\beta)L_{x,0}
					\sum_{t=0}^{T-1}\sum_{i=0}^{t}
					\beta^{t-i}
					\mathbb E
					\left[
					\|e_i\|\mathbf 1_{\{i<\tau\}}
					\right]}_{=:I_{2,\tau}}                                    \\
				&+
				\underbrace{
					(1-\beta)L_{x,1}
					\sum_{t=0}^{T-1}\sum_{i=0}^{t}
					\beta^{t-i}
					\mathbb E
					\left[
					\|\nabla\Phi(x^i)\|\|e_i\|
					\mathbf 1_{\{i<\tau\}}
					\right]}_{=:I_{3,\tau}}  \\
				&+
				\frac{\eta_xL_1\beta}{1-\beta}
				\sum_{t=0}^{T-1}
				\mathbb E
				\left[
				\|\nabla\Phi(x^t)\|
				\mathbf 1_{\{t<\tau\}}
				\right]                                                
				+
				\frac{\beta}{1-\beta}
				\|m^0-\nabla\Phi(x^0)\|.
			\end{aligned}
		\end{equation}

		For \(I_{1,\tau}\), 
		\begin{equation}\label{RS_equ_HP_I1tau_step1}
			\begin{aligned}
				I_{1,\tau}
				&\leq
				(1-\beta)
				\sum_{t=0}^{T-1}
				\left(
				\mathbb E
				\left\|
				\sum_{i=0}^{t}
				\beta^{t-i}
				\Delta_x^i
				\right\|^2
				\right)^{1/2} \left(\mathbb E\left[\mathbf 1_{\{t<\tau\}}\right]\right)^{1/2}                                            \\
				&\leq
				(1-\beta)
				\sum_{t=0}^{T-1}
				\left(
				\sigma_x^2
				\sum_{i=0}^{t}
				\beta^{2(t-i)}
				\right)^{1/2}                                             \\
				&\leq
				(1-\beta)\sigma_x\sqrt{T}
				\left(
				\sum_{t=0}^{T-1}
				\sum_{i=0}^{t}
				\beta^{2(t-i)}
				\right)^{1/2}                                             \\
				&\leq
				T\sigma_x\sqrt{1-\beta},
			\end{aligned}
		\end{equation}
		where the first inequality follows from the
		Cauchy--Schwarz inequality, the second inequality
		follows from Assumptions~\ref{RR_ass:unbiased} and~\ref{RS_ass_varianceBound}, and the third inequality
		follows from \(\sum_{i=1}^n \sqrt{a_i} \leq \sqrt{n} \sqrt{\sum_{i=1}^n a_i}.\)

		To bound \(I_{2,\tau}\), we show by induction that  the following inequality holds
		\begin{equation}\label{eq:stopped_geometric_moment}
			\mathbb E
			\left[
			\|e_i\|^2\mathbf 1_{\{i\leq\tau\}}
			\right]
			\leq
			\gamma^i\|e_0\|^2
			+
			\frac{C_1}{1-\gamma},\quad i=0,\ldots,T.
		\end{equation}
		
		Obviously, \eqref{eq:stopped_geometric_moment} holds for \(i=0\).
		Assume that \eqref{eq:stopped_geometric_moment} holds for some \(i\in\{0,\ldots,T-1\}\).
		
		Since
		\[
		\{i+1\leq\tau\}
		=
		\{i<\tau\}
		\in\mathcal F_i,
		\]
		\eqref{eq:onestep_stopped} gives
		\[
		\begin{aligned}
			&
			\mathbb E
			\left[
			\|e_{i+1}\|^2\mathbf 1_{\{i+1\leq\tau\}}
			\right]                                                    \\
			={}&
			\mathbb E
			\left[
			\mathbf 1_{\{i<\tau\}}
			\mathbb E
			\left[
			\|e_{i+1}\|^2
			\,\middle|\,
			\mathcal F_i
			\right]
			\right]                                                    \\
			\leq{}&
			\gamma
			\mathbb E
			\left[
			\|e_i\|^2\mathbf 1_{\{i<\tau\}}
			\right]
			+
			C_1\mathbb P(i<\tau)                                      \\
			\leq{}&
			\gamma
			\mathbb E
			\left[
			\|e_i\|^2\mathbf 1_{\{i\leq\tau\}}
			\right]
			+
			C_1                                                        \\
			\leq{}&
			\gamma^{i+1}\|e_0\|^2
			+
			\frac{C_1}{1-\gamma},
		\end{aligned}
		\]
		where the last inequality follows from the induction hypothesis. Thus, \eqref{eq:stopped_geometric_moment} holds.	
		
		Then, we have by Jensen's inequality and \eqref{eq:stopped_geometric_moment} that
		\[
		\begin{aligned}
			\mathbb E
			\left[
			\|e_i\|\mathbf 1_{\{i<\tau\}}
			\right]
			&\leq
			\left(
			\mathbb E
			\left[
			\|e_i\|^2\mathbf 1_{\{i\leq\tau\}}
			\right]
			\right)^{1/2}                                             \\
			&\leq
			\gamma^{i/2}\|e_0\|
			+
			\left(\frac{C_1}{1-\gamma}\right)^{1/2}.
		\end{aligned}
		\]
		Therefore,
		\begin{equation}\label{RS_equ_HP_I1tau_step2}
			\begin{aligned}
				I_{2,\tau}
				&\leq
				(1-\beta)L_{x,0}
				\sum_{t=0}^{T-1}\sum_{i=0}^{t}
				\beta^{t-i}
				\left(
				\gamma^{i/2}\|e_0\|
				+
				\left(\frac{C_1}{1-\gamma}\right)^{1/2}
				\right)                                                    \\
				&\leq
				L_{x,0}\|e_0\|
				\sum_{i=0}^{T-1}\gamma^{i/2}
				+
				TL_{x,0}\left(\frac{C_1}{1-\gamma}\right)^{1/2}                                        \\
				&\leq
				\frac{4L_{x,0}\|e_0\|}{\mu\eta_y}
				+
				TL_{x,0}\left(\frac{C_1}{1-\gamma}\right)^{1/2}.
			\end{aligned}
		\end{equation}

		For \(I_{3,\tau}\),
		\begin{equation}\label{RS_equ_HP_I1tau_step3}
			\begin{aligned}
				I_{3,\tau}
				&\leq
				\frac14
				\sum_{i=0}^{T-1}
				\left(
				(1-\beta)
				\sum_{t=i}^{T-1}\beta^{t-i}
				\right)
				\mathbb E
				\left[
				\|\nabla\Phi(x^i)\|
				\mathbf 1_{\{i<\tau\}}
				\right]                                                    \\
				&\leq
				\frac14
				\sum_{t=0}^{T-1}
				\mathbb E
				\left[
				\|\nabla\Phi(x^t)\|
				\mathbf 1_{\{t<\tau\}}
				\right],
			\end{aligned}
		\end{equation}
		where the first inequality follows from
		\(\|e_i\|\leq 1/(4L_{x,1})\) if \(i<\tau\).

		Substituting \eqref{RS_equ_HP_I1tau_step1}, \eqref{RS_equ_HP_I1tau_step2} and \eqref{RS_equ_HP_I1tau_step3} into \eqref{RS_equ_HP_I1tau_decomposition} yields
		\[
		\begin{aligned}
			&
			\sum_{t=0}^{T-1}
			\mathbb E
			\left[
			\|\varepsilon_t\|\mathbf 1_{\{t<\tau\}}
			\right]                                                    \\
			\leq{}&
			T\sigma_x\sqrt{1-\beta}
			+\frac{4L_{x,0}\|e_0\|}{\mu\eta_y}
			+TL_{x,0}\left(\frac{C_1}{1-\gamma}\right)^{1/2}                                      \\
			&+
			\left(
			\frac14+\frac{\eta_xL_1\beta}{1-\beta}
			\right)
			\sum_{t=0}^{T-1}
			\mathbb E
			\left[
			\|\nabla\Phi(x^t)\|\mathbf 1_{\{t<\tau\}}
			\right]                                                    \\
			&+
			\frac{T\eta_xL_0\beta}{1-\beta}
			+\frac{\beta}{1-\beta}
			\|m^0-\nabla\Phi(x^0)\|.
		\end{aligned}
		\]
		Subsequently, \eqref{eq:stopped_descent} gives
		\[
		\begin{aligned}
			&
			C_0
			\frac1T
			\sum_{t=0}^{T-1}
			\mathbb E
			\left[
			\|\nabla\Phi(x^t)\|\mathbf 1_{\{t<\tau\}}
			\right]                                                    \\
			\leq{}&
			\frac{\Phi(x^0)-\Phi^*}{T\eta_x}
			+\frac{8L_{x,0}\|e_0\|}{T\mu\eta_y}
			+\frac{2\beta}{T(1-\beta)}
			\|m^0-\nabla\Phi(x^0)\|
			\\
			&+2L_{x,0}\left(\frac{C_1}{1-\gamma}\right)^{1/2}+2\sigma_x\sqrt{1-\beta}
			+\frac{\eta_xL_0}{2}
			+\frac{2\eta_xL_0\beta}{1-\beta},
		\end{aligned}
		\]
		where
		\[
		C_0
		:=
		\frac{1}{2}-\frac{\eta_xL_1}{2}
		-\frac{2\eta_xL_1\beta}{1-\beta}.
		\]
		Since \(\eta_x \leq \frac{1-\beta}{8L_1}\), it follows that
		\(C_0 \geq \frac14\).
		Substituting the parameter choices of the theorem into the preceding bound, we obtain
		\begin{equation*}
			\begin{aligned}
				&\frac{1}{4T}
				\sum_{t=0}^{T-1}
				\mathbb E
				\left[
				\|\nabla\Phi(x^t)\|\mathbf 1_{\{t<\tau\}}
				\right] \\                                                
				\leq{}&16\left(\Phi(x^0)-\Phi^*\right)\kappa L_{x,1}
				\sqrt{\frac{8L_{x,1}\sigma_y}{\mu}}
				\frac{1}{\delta^{3/4}T^{1/4}}+\frac{4L_{x,0}\sigma_y}
				{\mu\sqrt{T}}
				\\
				&+\frac{80L_{x,1}\sigma_y}{\mu}
				\frac{\|m^0-\nabla\Phi(x^0)\|}
				{\sqrt{\delta T}}+2L_{x,0}\left(\frac{\sqrt{3}}{8\sqrt{2}L_{x,1}}
				\frac{\sqrt{\delta}}{\sqrt{T}}
				+
				\sqrt{
					\frac{3\sigma_y}{8\mu L_{x,1}}
				}
				\frac{\delta^{1/4}}{T^{1/4}}\right)
				\\
				&
				+2\sigma_x
				\sqrt{
					\frac{\mu}{40L_{x,1}\sigma_y}
				}\frac{\delta^{1/4}}{T^{1/4}}
				+\frac{L_0}{32\kappa L_{x,1}}
				\sqrt{
					\frac{\mu}{8L_{x,1}\sigma_y}
				}
				\frac{\delta^{3/4}}{T^{3/4}}+\frac{5L_0\sigma_y}{\kappa\mu}
				\sqrt{
					\frac{\mu}{8L_{x,1}\sigma_y}
				}
				\frac{\delta^{1/4}}{T^{1/4}}.
			\end{aligned}
		\end{equation*}				
		Then,
		\begin{equation}\label{eq:localized_expected_stationarity}
			\frac1T
			\sum_{t=0}^{T-1}
			\mathbb E
			\left[
			\|\nabla\Phi(x^t)\|
			\mathbf 1_{\{t<\tau\}}
			\right]
			\leq
			\frac{C}{\delta^{3/4}T^{1/4}},
		\end{equation}
		where \(C>0\) is independent of \(T\) and \(\delta\).
		
		Combining \eqref{RS_equ_stopped_localization_failure} and \eqref{eq:localized_expected_stationarity} gives
		\[
		\begin{aligned}
			\frac1T
			\sum_{t=0}^{T-1}
			\mathbb E\|\nabla\Phi(x^t)\|
			&\leq
			\frac{C}{\delta^{3/4}T^{1/4}}
			+
			G_\Phi\delta.
		\end{aligned}
		\]
		Moreover, choosing \(\delta=T^{-1/7}\), the above bound reduces to
		\[
		\frac1T\sum_{t=0}^{T-1}
		\mathbb E\|\nabla\Phi(x^t)\|
		\leq
		C T^{-1/7}.
		\]
		The proof is complete.

	\end{proof}
	
	Inequality~\eqref{RS_equ_NSGDA-M_expectation_bound} in
	Theorem~\ref{RS_thm_NSGDA-M_expectation} consists of two terms.
	The first term \(C/(\delta^{3/4}T^{1/4})\) represents the convergence behavior
	within the localization region and corresponds to the localized
	expected stationarity bound established by
	Xian et al.~\cite{xian2024delving}. The second term
	\(G_\Phi\delta\) explicitly captures the effect of the localization
	failure event \(\{\tau<T\}\), which is controlled by
	\(\mathbb P(\tau<T)\leq\delta\) together with
	Assumption~\ref{RS_ass_boundedGradientIterates}. The second term
	cannot be omitted when establishing convergence in expectation.

	\section{Numerical Experiments}\label{RS_Section_4}
	To evaluate the performance of NSGDA-M, we conduct experiments on a real-world application of distributionally robust logistic regression~\cite{xian2024delving}, which is formulated as follows
	\begin{equation}\label{euq:experimet_real}
		\min _{x \in \mathbb{R}^{n}} \max _{y \in Y \subset \mathbb{R}^{N}}\mathcal{L}(x, y)\coloneqq \frac{1}{N} \sum_{i=1}^{N} y_{i}\ell(x;a_{i},b_{i})+f(x)-g(y),
	\end{equation}
	with $$\ell(x;a_{i},b_{i})=\log \left(1+\exp \left(-b_{i} (a_{i})^{\top} x\right)\right),$$
	$$f(x)=\lambda_{1} \sum_{i=1}^{n} \frac{\alpha x_{i}^{2}}{1+\alpha x_{i}^{2}},\,\,
	g(y)=\frac{1}{2} \lambda_{2}\left\|N y-\mathbf{1}\right\|^{2}.$$
	Here $\ell(\cdot)$ is the logistic loss function,
	$f(\cdot)$ is a nonconvex regularizer, and $g(\cdot)$ is a distributionally robust regularizer. The set
	$Y \coloneqq\left\{y \in \mathbb{R}_{+}^{N}: \mathbf{1}^{\top} y=1\right\}$ is a simplex, where $\mathbf{1}$ denotes the all-ones vector, and $\mathcal{L}(\cdot)$ is the average loss function over the entire dataset.
	In problem \eqref{euq:experimet_real}, $x$ is the classifier parameter, $N$ is the  number of training samples, $a_{i}\in\mathbb{R}^{n}$ is the feature vector of the $i$-th sample, and $b_{i}\in\{-1,1\}$ is its label.
	
	We conduct experiments on nine benchmark binary classification 
	datasets (\texttt{a9a}, \texttt{covtype}, \texttt{diabetes}, 
	\texttt{german}, \texttt{gisette}, \texttt{ijcnn1}, 
	\texttt{mushrooms}, \texttt{phishing}, \texttt{w8a}) from 
	the LIBSVM repository\footnote{\url{https://www.csie.ntu.edu.tw/~cjlin/libsvmtools/datasets}}.
	Table \ref{RR_tab_datasets} reports the number of samples and features in each dataset.
	Following \cite{xian2024delving}, we set the parameters in problem~\eqref{euq:experimet_real} to $\lambda_1=0.001$, $\lambda_2=\frac{1}{N^2}$ and $\alpha=10$. 
	\begin{table}[htbp]
		\centering
		\small
		\setlength{\tabcolsep}{3.5pt}
		\renewcommand{\arraystretch}{1.15}
		\begin{tabular}{@{}lccccccccc@{}}
			\toprule
			Name & a9a & covtype & diabetes & german & gisette & ijcnn1 & mushrooms & phishing & w8a \\
			\midrule
			Samples & 32561 & 581012 & 768 & 1000 & 6000 & 141691 & 8124 & 11055 & 49749 \\
			Features & 123 & 54 & 8 & 24 & 5000 & 22 & 112 & 68 & 300 \\
			\bottomrule
		\end{tabular}
		\caption{Characteristics of the LIBSVM binary classification datasets}
		\label{RR_tab_datasets}
	\end{table}
	
	We compare NSGDA-M with the normalized stochastic gradient descent ascent (NSGDA) and stochastic gradient descent ascent (SGDA) methods with constant stepsizes in \cite{xian2024delving}. 
	For NSGDA-M, the batch size is $1$, and the momentum parameter $\beta$ is selected from $\{0.1,0.4,0.7,0.9\}$ by grid search. For NSGDA and SGDA, following the setting of \cite{xian2024delving}, the batch size is set to be $50$. 
	For each algorithm, the stepsizes $\eta_x$ and $\eta_y$
	are selected from $\{0.1, 0.01, 0.001, 10^{-4}, 10^{-5}, 10^{-6}\}$ by grid search. 
	We report the convergence behavior of the algorithms in Figure~\ref{fig:converge}, which depicts the norm of gradient of the primal function versus the number of iterations. As observed from Figure~\ref{fig:converge}, the convergence performance 
	of NSGDA-M is comparable to that of NSGDA on most datasets; moreover, NSGDA-M 
	exhibits more stable convergence behavior. In contrast, SGDA demonstrates 
	inferior convergence performance across the majority of datasets.
	
	\begin{figure}[H]
		\centering
		\subfloat[a9a]{%
			\includegraphics[width=0.32\textwidth]{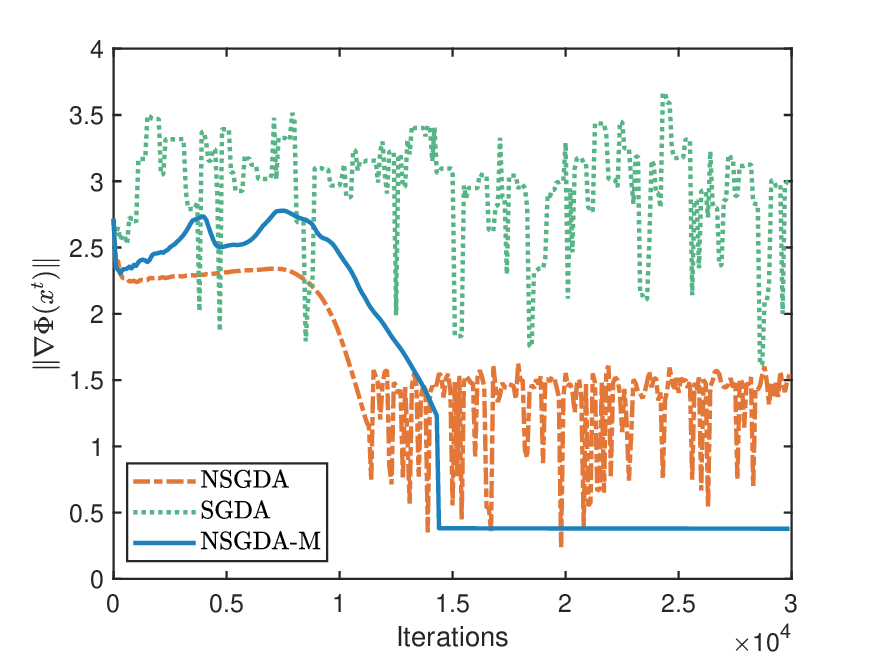}%
		}\hfill
		\subfloat[covtype]{%
			\includegraphics[width=0.32\textwidth]{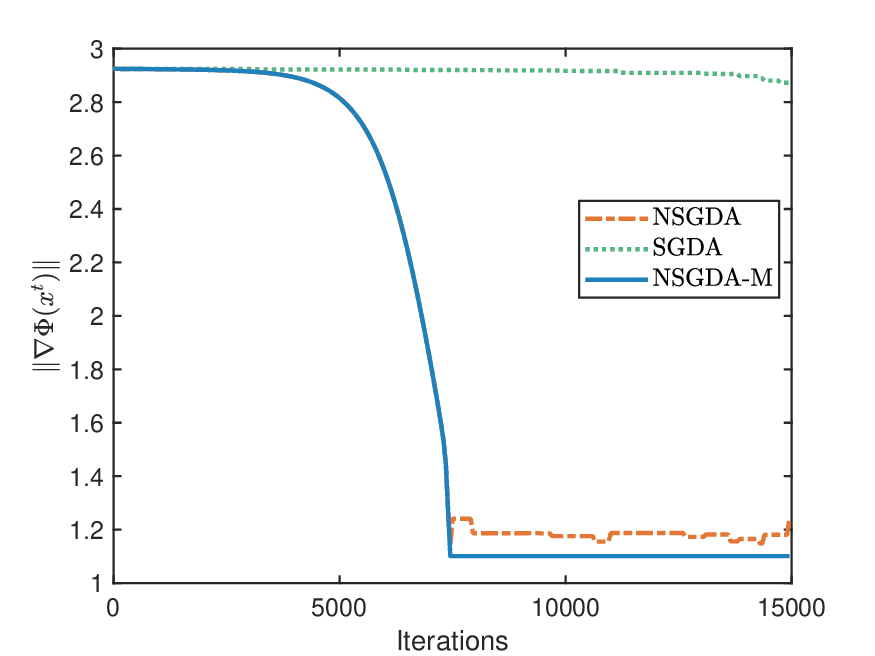}%
		}\hfill
		\subfloat[diabetes]{%
			\includegraphics[width=0.32\textwidth]{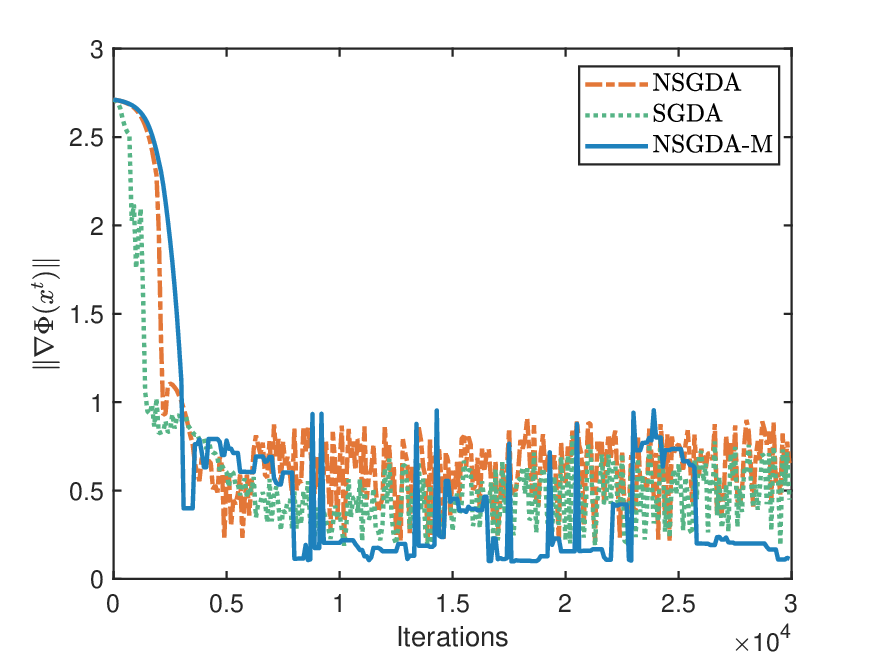}%
		}
		\par
		
		\subfloat[german]{%
			\includegraphics[width=0.32\textwidth]{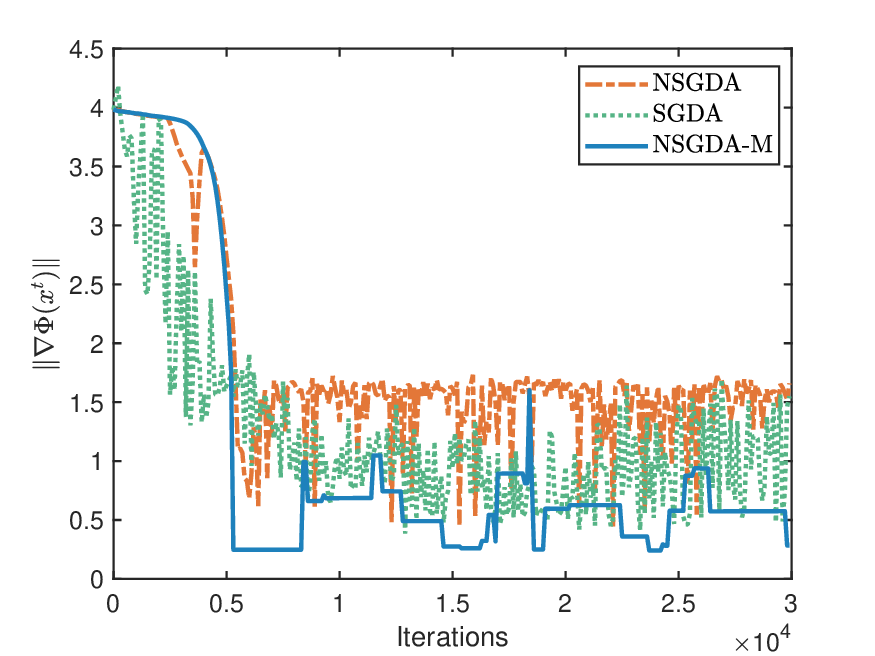}%
		}\hfill
		\subfloat[gisette]{%
			\includegraphics[width=0.32\textwidth]{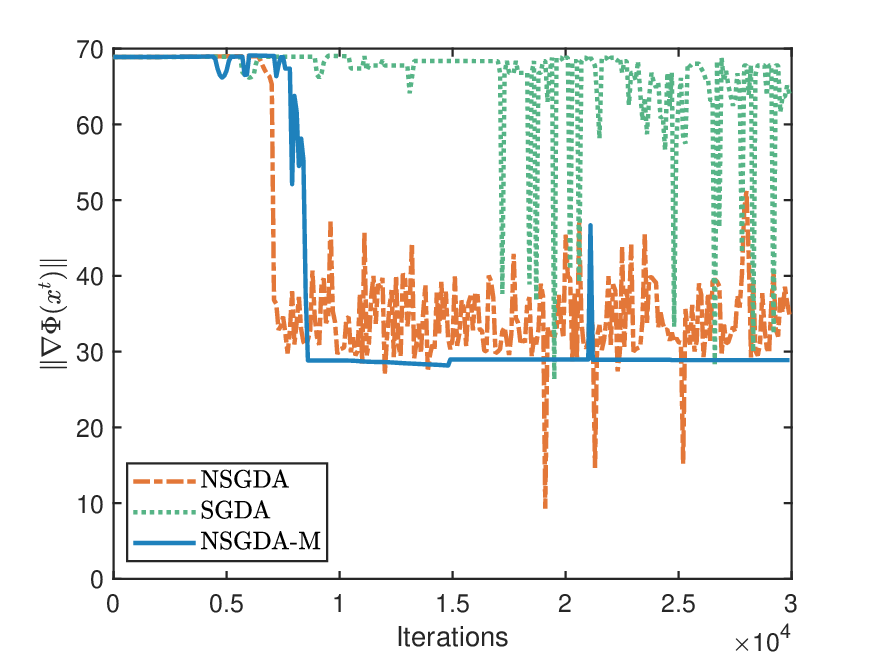}%
		}\hfill
		\subfloat[ijcnn1]{%
			\includegraphics[width=0.32\textwidth]{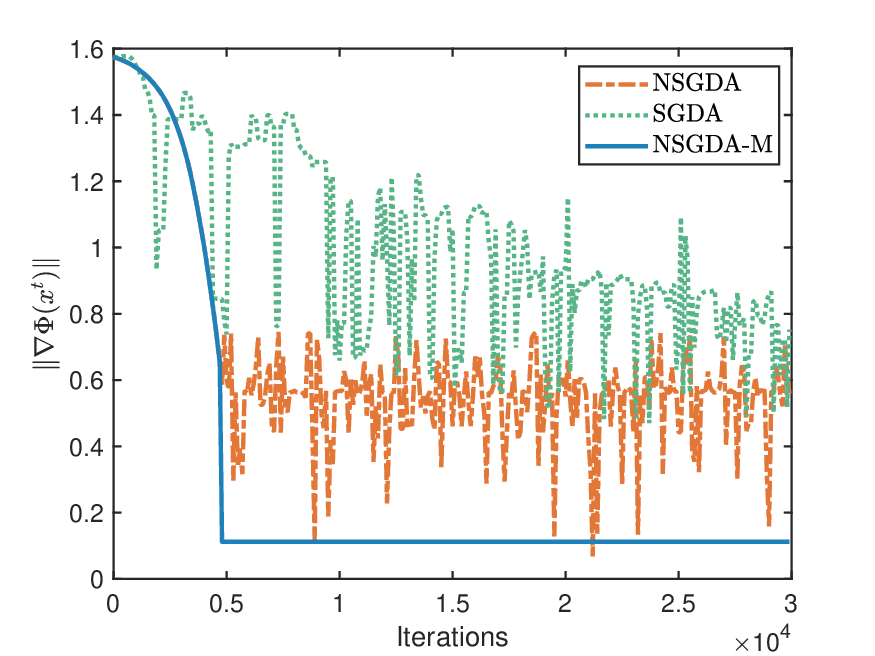}%
		}
		\par
		
		\subfloat[mushrooms]{%
			\includegraphics[width=0.32\textwidth]{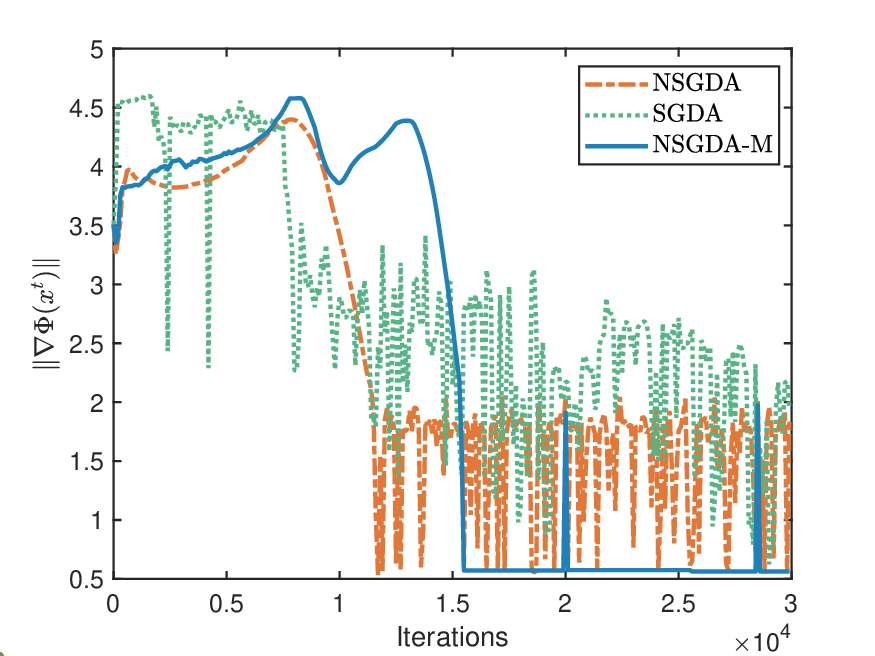}%
		}\hfill
		\subfloat[phishing]{%
			\includegraphics[width=0.32\textwidth]{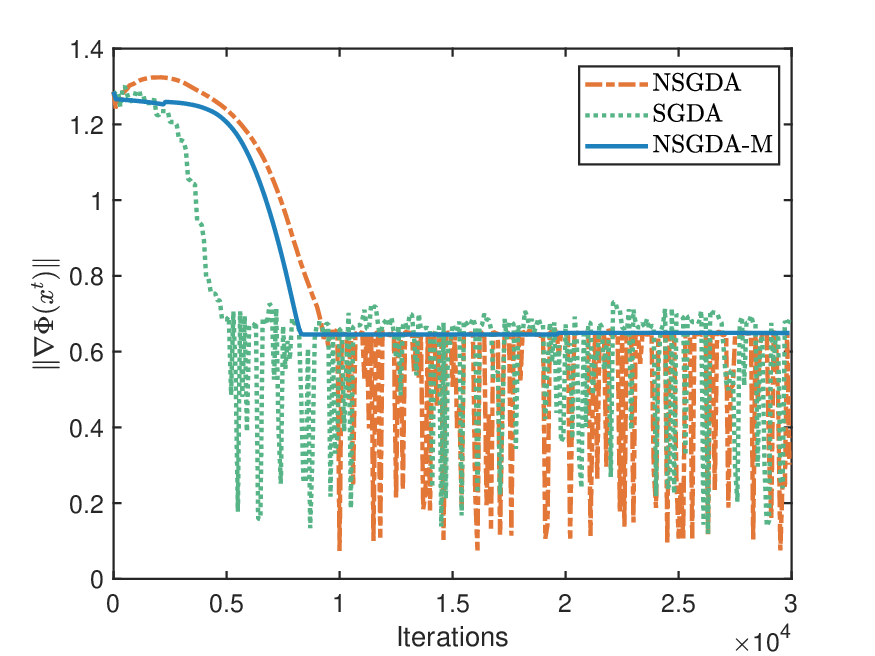}%
		}\hfill
		\subfloat[w8a]{%
			\includegraphics[width=0.32\textwidth]{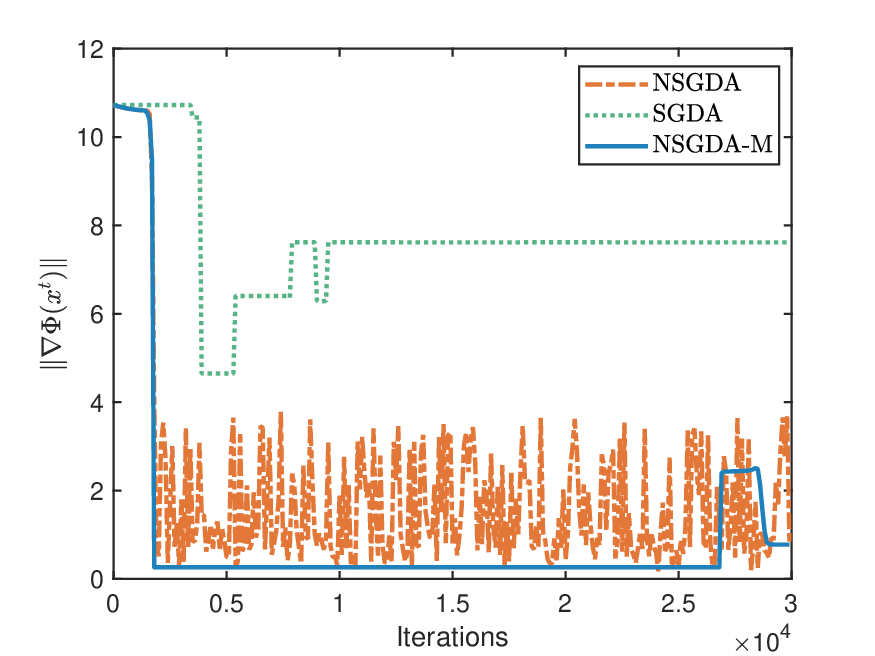}%
		}
		
		\caption{Convergence of NSGDA-M, NSGDA and SGDA}
		\label{fig:converge}
	\end{figure}

	\noindent\textbf{Funding}
	This research is supported by the National Natural Science
	Foundation of China (Grant No. 12471283) and the Fundamental Research
	Funds for the Central Universities (Grant No. DUT24LK001).
	\medskip
	
	\noindent\textbf{Data availability}
	No new datasets were generated in this study. The datasets
	used in the numerical experiments are publicly available from their
	original sources.
	
	\medskip
	
	\noindent\textbf{Conflict of interest}
	The authors report no potential conflict of interest.
	
	\bibliographystyle{plain}
	\bibliography{references}
	
\end{document}